\documentclass[pdflatex,sn-mathphys-num]{sn-jnl}

\usepackage{graphicx}%
\usepackage{multirow}%
\usepackage{amsmath,amssymb,amsfonts}%
\usepackage{amsthm}%
\usepackage{mathrsfs}%
\usepackage[title]{appendix}%
\usepackage{xcolor}%
\usepackage{textcomp}%
\usepackage{manyfoot}%
\usepackage{booktabs}%
\usepackage{algorithm}%
\usepackage{algorithmicx}%
\usepackage{algpseudocode}%
\usepackage{listings}%
\usepackage{array}

\usepackage{ytableau}
\usepackage{mathtools}
\usepackage{thm-restate}
\usepackage[shortlabels,inline]{enumitem}
\usepackage{tikz}
\usetikzlibrary{calc}
\usepackage[capitalise,nameinlink]{cleveref}

\theoremstyle{thmstyleone}%

\newtheorem{theorem}{Theorem}[section]
\newtheorem{lemma}[theorem]{Lemma}

\newtheorem{proposition}[theorem]{Proposition}
\newtheorem{corollary}[theorem]{Corollary}
\newtheorem{question}[theorem]{Question}

\theoremstyle{thmstyletwo}%
\newtheorem{example}[theorem]{Example}
\newtheorem{remark}[theorem]{Remark}

\theoremstyle{thmstylethree}%
\newtheorem{definition}[theorem]{Definition}

\numberwithin{equation}{section}

\newcommand{\N}{\mathbb{N}}

\newcommand{\Q}{\mathbb{Q}}
\newcommand{\R}{\mathbb{R}}
\DeclarePairedDelimiter{\floor}{\lfloor}{\rfloor}
\DeclarePairedDelimiter{\ceil}{\lceil}{\rceil}
\DeclareMathOperator{\Par}{Par}
\DeclareMathOperator{\pet}{pet}
\DeclareMathOperator{\Pet}{Pet}
\DeclareMathOperator{\core}{core}
\DeclareMathOperator{\htt}{ht}
\newcommand{\inv}{\mathsf{inv}}
\newcolumntype{C}[1]{>{\centering\arraybackslash}p{#1}}

\begin{document}

\title[Combinatorial proofs of Petrie Pieri rule and Plethystic Pieri rule]{Combinatorial proofs of Petrie Pieri rule and Plethystic Pieri rule}


\author*[1]{\fnm{Saintan} \sur{Wu}}\email{saw036@ucsd.edu}

\author[2,3]{\fnm{Sen-Peng} \sur{Eu}}\email{speu@math.ntnu.edu.tw}
\equalcont{These authors contributed equally to this work.}

\author[4]{\fnm{Kuo-Han} \sur{Ku}}\email{r11221008@ntu.edu.tw}
\equalcont{These authors contributed equally to this work.}

\author[4]{\fnm{Yu-Sheng} \sur{Shih}}\email{r13922048@ntu.edu.tw}
\equalcont{These authors contributed equally to this work.}

\affil*[1]{\orgdiv{Department of Mathematics}, \orgname{UC San Diego}, \orgaddress{\street{9500 Gilman Drive}, \city{La Jolla}, \postcode{92092}, \state{California}, \country{USA}}}

\affil[2]{\orgdiv{Department of Mathematics}, \orgname{National Taiwan Normal University}, \orgaddress{\city{Taipei}, \postcode{11677}, \state{Taiwan}, \country{ROC}}}

\affil[3]{\orgname{Chinese Air Force Academy}, \orgaddress{ \city{Kaohsiung}, \postcode{82047}, \state{Taiwan}, \country{ROC}}}

\affil[4]{\orgdiv{Department of Mathematics}, \orgname{National Taiwan University}, \orgaddress{\street{No.1, Sec. 4, Roosevelt Road}, \city{Taipei}, \postcode{106319}, \state{Taiwan}, \country{ROC}}}

\date{Nov 11, 2025}



\abstract{Petrie symmetric functions $G(k,n)$, also known as truncated homogeneous symmetric functions or modular complete symmetric functions, form a class of symmetric functions interpolating between the elementary symmetric functions $e_n$ and the homogeneous symmetric functions $h_n$. Analogous to the Pieri rule for $s_\mu h_n$ and the dual Pieri rule for $s_\mu e_n$, Grinberg showed that the Schur coefficients for the ``Pieri rule'' of $s_\mu G(k,n)$ can be determined by the determinant $\pet_k(\lambda,\mu)$ of Petrie matrices. Cheng, Chou, Eu, Fu, and Yao provided a ribbon tiling interpretation for the coefficient $\pet_k(\lambda,\varnothing)$.

We offer a more transparent combinatorial perspective on the structure and behavior of Petrie symmetric functions. First, we provide a weighted formula for the determinant of a Petrie matrix in terms of certain orientations of the associated Petrie graph derived from the matrix. We then give a formula for $\pet_k(\lambda,\mu)$ for all skew shapes using just combinatorics. In addition, we investigate the generating function of these orientations with respect to certain statistics. As an application of our method, we present a combinatorial proof of the plethystic Pieri rule.

As a side discovery, we find a connection between $2$-colored noncrossing set partitions and $1234$-avoiding permutations.}


\maketitle

\section*{Declarations}

The authors did not receive support from any organization for the submitted work. The authors have no competing interests to declare that are relevant to the content of this article.

\section{Introduction}\label{sec:intro}

\subsection{Petrie symmetric functions}
	The Petrie symmetric functions $G(k,n)$, also known as truncated homogeneous symmetric functions in \cite{fu2020truncated} and modular complete symmetric functions in \cite{Walker1994ModularSF}, are a class of symmetric functions defined as 
	\[G(k,n):= \sum_{\lambda: \lambda\vdash n, \lambda_1<k} m_\lambda,\]
	where $m_\lambda$ is the monomial symmetric function indexed by the partition $\lambda$.
	In formal power series form, we can also write this as
	\[\sum_{n\geq 0} G(k,n)t^n = \prod_{i=1}^\infty (1+x_it+x_i^2t^2+\cdots+x_i^{k-1}t^{k-1}).\]
	
	By the definition, $G(2,n)=e_n$ and $G(k,n)=h_n$ if $k>n$. Hence, the Petrie symmetric functions also serve as an interpolation between elementary symmetric functions $e_n$ and homogeneous symmetric functions $h_n$. 
	
	The expansion of $G(k,n)$ to Schur symmetric functions has drawn a great deal of attention in the literature. Liu and Polo \cite{LIU2021105352} first showed that for prime $p$ we have the expansion 
    \[G(p,p)=\sum_{i=0}^{p-2} (-1)^is_{(p-1-i,1^{i+1})} \qquad \text{and} \qquad G(p,2p-1)=\sum_{i=0}^{p-2} (-1)^is_{(p-1,p-1-i,1^{i+1})}.\]
    Grinberg \cite{Grinberg_2022} then gives an expansion formula
		\[ s_\mu\cdot G(k,n) = \sum_{\lambda \supset \mu, |\lambda|=n+|\mu|} \pet_k(\lambda,\mu) s_\lambda\]
	for any $k\geq 0$ and any partition $\mu$, where the \emph{$k$-Petrie number} $\pet_k(\lambda,\mu)$ of $\lambda,\mu$ is defined as
	\[\pet_k(\lambda,\mu):=\det(\Pet_k(\lambda,\mu)):=\det\begin{pmatrix}
		[0\leq \lambda_i-i-\mu_j+j<k]_{1\leq i,j\leq \ell(\lambda)}
	\end{pmatrix}.\]
	Here we use the Iverson symbol $[P]:= 1$ if $P$ is true and $[P]:=0$ if otherwise. We list some basic facts about the $k$-Petrie numbers:
	\begin{proposition}\label{prop:petk}
		Given two partitions $\lambda,\mu$. 
		\begin{enumerate}
			\item If $\mu\not\subset \lambda$ or $\lambda_i \geq \mu_i+k$ for some $i$, then $\pet_k(\lambda,\mu)=0$. In general, we can restrict our discussion to skew shapes $\lambda/\mu$.
			\item Though $\Pet_k(\lambda,\mu)$ might vary, $\pet_k(\lambda,\mu)$ only depend on the shape of $\lambda/\mu$. In particular, trailing zeros of $\lambda$ and $\mu$ do not change $\pet_k(\lambda,\mu)$.
			\item If $\lambda/\mu$ can be decomposed into $2$ skew shapes $\lambda^1/\mu^1,\lambda^2/\mu^2$, then $\pet_k(\lambda,\mu)=\pet_k(\lambda^1/\mu^1)\pet_k(\lambda^2/\mu^2)$.
		\end{enumerate}
	\end{proposition}

	In the case when $\mu=\varnothing$, Grinberg further gave a formula to compute $\pet_k(\lambda,\varnothing)$ explicitly. Using Grinberg's formula, Cheng et al.  \cite{cheng2022signedmultiplicitiesschurexpansions} gave a combinatorial interpretation of $\pet_k(\lambda,\varnothing)$:

	\begin{theorem}
[Cheng et al.\cite{cheng2022signedmultiplicitiesschurexpansions}]
		\label{thm:Cheng-unused}
		For $k\geq 1$ and $\lambda \vdash n$. the value of $\pet_k(\lambda,\varnothing)$ is given as follows:
		\begin{enumerate}
			\item If $\lambda_{1} \geq k$, then $\pet_k(\lambda,\varnothing)=0$.
			\item Let $\lambda_{1}<k$. If $\core_{k}(\lambda)$ has more than one part, then $\pet_k(\lambda,\varnothing)=0$; 
            otherwise, let $q=\floor*{\frac{n}{k}}$ and let $(\lambda^{0}, \lambda^{1}, \ldots, \lambda^{q})$ be any ribbon tiling of $\lambda/\core_k(\lambda)$, i.e., a sequence of $k$-partitions such that
		
			\[\operatorname{core}_{k}(\lambda)=\lambda^{0} \subset \lambda^{1} \subset \cdots \subset \lambda^{q}=\lambda\]
			and $\lambda^{j} / \lambda^{j-1}$ is a ribbon of size $k$ for each $j \in\{1, \ldots, q\}$. Then
			
			\[
			\operatorname{pet}_{k}(\lambda,\varnothing)=\prod_{j=1}^{q}(-1)^{\mathrm{ht}\left(\lambda^{j} / \lambda^{j-1}\right)+1},
			\]
			where we use the convention that an empty product equals $1$.
		\end{enumerate}
	\end{theorem}
	
    In 2024, Jin, Jing and Liu \cite{jin2024pierilikerulepetriesymmetric} suggested a combinatorial formula to $\pet_k(\lambda,\mu)$ when $\lambda/\mu$ is connected, in terms of tilings using $k$-ribbons and one horizontal strip. Although their proof seemed to be flawed, it turns out to be the right object of interpreting $\pet_k(\lambda,\mu)$.

    We may try to decompose $\lambda/\mu$ into ribbons $\Theta_1,\dotsc,\Theta_q$ of length $k$ and a horizontal strip $\nu/\mu$, so that the bottom-leftmost cell of each ribbon is not to the right of other ribbons. 
    If it exists, we call such a decomposition $(\nu,\Theta)$ of $\lambda/\mu$ a \emph{good pair} (or a $k$-good pair), where $\Theta=(\Theta_1,\dotsc,\Theta_q)$. 



\subsection{Petrie matrices and graphs}

Petrie symmetric functions got its name because their Schur expansion comes from the determinant of a special kind of matrices called \textit{Petrie matrices}, named after the archaeologist Flinders Petrie who developed the method of sequence dating in archaeology.

	\begin{definition}\mbox{}
		\begin{enumerate}
			\item Let $v[i,j]$ be the row vector that has $1$ at the columns $i+1,\dotsc,j$ and $0$ at the others, i.e., $(\underbrace{0,0,\dotsc,0}_{i \text{ zeros}},\underbrace{1,1,\dotsc,1}_{(j-i)\text{ ones}},\underbrace{0,0,\dotsc,0}_{(n-j)\text{ zeros}})$. Also, let $v[j,i]=-v[i,j]$ when $j>i$. 
			\item A \emph{Petrie matrix} is a $n\times n$ matrix with all its rows in the form $v[i,j]$ for some $0\leq i\leq j\leq n$.
			\item Given an $n\times n$ Petrie matrix $A$ with rows $v[i_1,j_1],v[i_2,j_2],\dotsc,v[i_n,j_n]$, we define its \emph{Petrie graph} $\mathcal{P}(A)$ to be a graph with vertices $\{0,1,\dots,n\}$, with (directed) edges $(i_1,j_1),(i_2,j_2),\dotsc,(i_n,j_n)$ labelled by $1,2,\dotsc,n$, respectively.\footnote{The vertex labels are better understood as half-integers $i+\frac 12$.}
		\end{enumerate}
	\end{definition}
	Note that all $v[i,i]$ give the same zero vector, so the Petrie graph may not be uniquely defined unless the indices $i,j$ of $v[i,j]$ are explicitly given, and this will always be the case in this paper. The determinant will always be $0$ in this case, regardless. 

    In 1974, Gordon and Wilkinson~\cite{Gordon1974DeterminantsOP} defined the Petrie graphs and used them to show that a square Petrie matrix has determinant $0$ if the graph is not a tree, and $\pm 1$ otherwise. As a corollary, $\pet_k(\lambda, \mu)$ is always $0$ or $\pm 1$. However, this algorithmic criterion is not convenient to calculate the expansion of $s_\mu\cdot G(k,n)$. 

\subsection{Main results and paper structure}

\medskip
In \cref{sec:Symbols}, preliminary background and definitions are provided.  
In \cref{sec:WeightSum}, we characterize the determinant of a Petrie matrix in terms of what we called \textit{good orientations}. Given any graph $G$ with ordered vertices $0,1,\dotsc,n$ and unordered edges $\{a_1,b_1\},\dotsc,\{a_m,b_m\}$. If $m=n$, we say an edge orientation of $G$, i.e., an assignment $(a,b)$ or $(b,a)$ to each edge $\{a,b\}$, is good if all but the last vertex $n$ receive exactly one outgoing edge, whereas the last one receives none.

For a good orientation $\sigma=\{(a_i,b_i): 1\leq i\leq n\}$, we define its \emph{order} and \emph{inversion} by
\[
|\sigma| = |\{ i : a_i > b_i \}|, \quad 
\mathsf{inv}(\sigma) = \mathsf{inv}(a_1 a_2 \dots a_n) = |\{ (i, j) : i < j \text{ and } a_i > a_j \}|,
\]
regarding $a_1 a_2 \dots a_n$ as a word. Using these terms, we prove

\begin{theorem}\label{thm:detPet_goodOrientation}
Let $A$ be a Petrie matrix with Petrie graph $\mathcal{P}(A)$. Then
\[
\det(A) = \sum_{\sigma:\,\text{good orientation of $\mathcal{P}(A)$}} (-1)^{|\sigma| + \mathsf{inv}(\sigma)}.
\]
\end{theorem}

\medskip
In \cref{sec:Graph}, we describe the Petrie graphs $\mathcal{P}(\Pet_k(\lambda,\mu))$ in terms of partitions $\lambda$, $\mu$, and build up a dictionary among skew shapes, Petrie matrices, abaci, and Petrie graphs. In particular, we prove that the good orientations are in correspondence with the good pairs. Together with \cref{thm:detPet_goodOrientation}, we give a refined story for $\pet_k(\lambda,\mu)$ in \cref{sec:Ribbon}.

\begin{theorem}
\label{thm:refinedJJL}
		Let $\lambda/\mu$ be a skew shape, not necessarily connected. Then either
\begin{enumerate}
\item there is a unique good pair ($\nu,\Theta$). In this case, $\pet_k(\lambda,\mu)=\pm 1$ and each component of $\nu/\mu$ has fewer than $k$ blocks; or
\item there are $0$ or $2^m$ good pairs for some $m\geq 1$. In this case, $\pet_k(\lambda,\mu)=0$.
\end{enumerate}
In either case, we have
		\[\pet_k(\lambda,\mu) = \sum_{(\nu,\Theta):\text{good}} \left(\prod_{i=1}^{q} (-1)^{\htt(\Theta_i)+1}\right).\]
\end{theorem}

\medskip
Note that good pairs $(\nu,\Theta)$ can have different number of ribbons in $\Theta$, and the power-of-2 phenomenon does not hold when we fix $|\Theta|$. The true story is binomial. Consider the generating functions of the good orientations \[D_k(\lambda/\mu; t,q)=\sum_{\sigma\,: \text{good orientation of }\mathcal{P}(\Pet_k(\lambda,\mu))} t^{|\sigma|} q^{\mathsf{inv}(\sigma)}.\] Note that \cref{thm:detPet_goodOrientation} says $\pet_k(\lambda,\mu)=D_k(\lambda/\mu; -1,-1)$. 

With the dictionary between good orientations and good pairs, we factor $D_k(\lambda/\mu; t,q)$ nicely in \cref{sec:GenFunc}, :
\begin{theorem}\label{thm:binomial}
        \[
            D_k(\lambda/\mu; t,q) = t^aq^b (1+q)^{\mathsf{c}_k(\lambda, \mu)} (1+t)^{\mathsf{nc}_k(\lambda, \mu)},
            \]
        where the statistics $a,b$, $\mathsf{c}_k(\lambda, \mu)$, and $\mathsf{nc}_k(\lambda, \mu)$ is defined in the proof in \cref{sec:GenFunc}.
\end{theorem}
Consequently, let $D_{k,n}(\lambda/\mu)$ denote the number of good pairs of $\lambda/\mu$ with $|\Theta|=n$. Then $\sum_{n\geq 0} D_{k,n}(\lambda,\mu)t^n = D_k(\lambda,\mu; t,1)$, and $D_{k,n}=2^{\mathsf{c}_k(\lambda, \mu)}\binom{\mathsf{nc}_k(\lambda, \mu)}{n-a}$ for $a\leq n\leq a+\mathsf{nc}_k(\lambda, \mu)$, and $0$ otherwise. 
\medskip

In \cref{sec:Plethysm}, we apply our method to give another combinatorial proof of the \emph{Plethystic Pieri rule}, first proved by Thibon et al.~\cite{Desarmenien1994}. For a symmetric function $f(x_1, x_2, \dots)$, let $f \circ p_k := f(x_1^k, x_2^k, \dots)$ be the \emph{plethysm} of $f$ and the power-sum symmetric function $p_k$. The dual Plethystic Pieri rule states:

\begin{theorem} \label{thm:plethysticPieri}
Let $\mu$ be a partition and $n, k \in \mathbb{N}$. Then
\[(e_n \circ p_k)\, s_\mu = \sum_{\lambda} (-1)^{\sum_{i=1}^n \text{ht}(\Theta_i)} s_\lambda,\]
where $\lambda$ runs over the partitions obtained from $\mu$ by adding $n$ $k$-ribbons $\Theta_1, \Theta_2, \dots, \Theta_n$ that are left-justified, i.e., if $(\mu,(\Theta_1,\dotsc,\Theta_n))$ is a good pair of $\lambda/\mu$ with trivial horizontal strip. 
\end{theorem}
Note that different left-justified tilings give different $\lambda$.\footnote{One can either check this by checking the number of boxes at each diagonal, or by noticing that we have to keep removing ribbons from the bottom-leftmost box of $\lambda/\mu$ to get the tiling.} We may as well sum over possible ways of adding $n$ left-justified $k$-ribbons.

\begin{remark}
    The Plethystic Pieri rule addresses $(h_n\circ p_k)\cdot s_\mu$ using top-justified ribbons. In some texts, this pair of rules is referred to as \textit{Plethystic Murnaghan-Nakayama rules} by considering them a generalization of Murnaghan-Nakayama rule.
\end{remark}

In the concluding section \cref{sec:Remark}, we study a certain type of Petrie matrices that we called \emph{Cycle Petrie matrix} and describe them with a reflection rule. We also pose a conjecture of the red-blue Motzkin path defined in \cref{sec:GenFunc} and describe its relation between $2$-colored noncrossing set partitions and $1234$-avoiding permutations. 
\bigskip

\section{Notation and Preliminaries}\label{sec:Symbols}
	In this section, we introduce symbols for Young diagrams, ribbons, and symmetric functions, as well as some basic facts about them. Most of the symbols follow from \cite{macdonald1998symmetric}.
	\subsection{Young diagrams}
	A \emph{partition} $\lambda$ of $n$ is a sequence $(\lambda_1,\lambda_2,\dotsc,\lambda_\ell)$ of integers such that $\lambda_1\geq \lambda_2\geq\cdots\geq \lambda_\ell> 0$ and $\lambda_1+\lambda_2+\cdots+\lambda_\ell=n$. Each entry $\lambda_i$ is a \emph{part} of $\lambda$. We write $\lambda \vdash n$ to show $|\lambda|=n$ (\emph{size} of $\lambda$) and $\ell(\lambda)=\ell$ (\emph{length} of $\lambda$). For the index out of bounds, we define $\lambda_{i}=0$ for $i>\ell$. Let $\Par$ be the set of all partitions.
	
	For partitions $\lambda$, we associate the \emph{Young diagram/shape}, also denoted as $\lambda$, with a left-justified array of cells with $\lambda_i$ cells in the $i$-th row, presented in English Notation. For partition $\lambda,\mu$, we define $\mu\subseteq \lambda$ if $\mu_j\leq \lambda_j$ for every $j$. The conjugate of $\lambda$ is a partition whose diagram is a reflection of shape $\lambda$ by the main diagonal, denoted by $\lambda'$. 
	
	\ytableausetup{smalltableaux}
	\begin{figure}[htbp!]
		\centering
		\[\begin{array}{cccc}
			\ydiagram{6,4,4,2,1} & \ydiagram{5,4,3,3,1,1}& \ydiagram[*(white) ]
			{6,3,2,2,1}
			*[*(gray)]{0,4,4}
			& \ydiagram[*(white) ]
			{3,1,1}
			*[*(gray)]{6}
			*[*(gray)]{0,3,2}
			*[*(gray!50)]{0,4,4}
			*[*(gray!50)]{0,0,0,2,1}
			\\
			(6,4,4,2,1) & (5,4,3,3,1,1) & (4,4,4)/(4,3,2) & (3,1,1)\\
			\lambda & \lambda' & \text{a 3-ribbon of }\lambda & \core_3(\lambda)\\
		\end{array}\]
		\caption{Young diagram, the conjugate, a $3$-ribbon, and the $3$-core of $\lambda=(6,4,4,2,1)$.}
		\label{fig:definition}
	\end{figure}
	
	The diagram of a \emph{skew shape} $\lambda/\mu$, where $\mu\subseteq\lambda$, is obtained by removing $\mu$ from $\lambda$ at the north-west corner. A skew shape is a \emph{horizontal strip} if its shape contains at most one box in each column.
	
	We say $\lambda/\mu$ is a \emph{ribbon} (or \emph{border strip} or \emph{rim hook}) if it is a connected diagram with no $2\times 2$ square. The \emph{height} $\htt(\lambda/\mu)$ of a ribbon is defined to be the number of row borders it crosses, or one less than the number of rows that $\lambda/\mu$ occupies. 

    The \emph{$k$-core} $\core_k(\lambda)$ is obtained from $\lambda$ by successively removing all ribbons of size $k$. We note the fact that $\core_k(\lambda)$ is unique regardless of the $k$-ribbons removed. This can be seen easily using abacus notation defined in the next subsection.
    
    \subsection{Maya diagrams, Abaci and runners}
	The \emph{content} of a grid $(i,j)$ in the diagram is $j-i$. The \emph{Maya diagram} of a Young diagram $\lambda$ is a two-way infinite sequence 
	\[a_i= \begin{cases}
		1, & i= \lambda_j-j+1 \text{ for some }j\geq 1;\\
		0, & \text{otherwise,}
	\end{cases}\]

    \begin{remark}
        Some authors use half integers $n-\frac12$ for the indices of $a_n$ in Maya diagrams. Here, we our theory works best with ceiling as we will correspond $a_{-i}$ (or $a_{-i-\frac 12}$) to vertex $i$ (which shall be understood as $i+\frac 12$) in Petrie graphs.
    \end{remark}
    
	When we place the Maya diagram along the outer border of $\lambda$, $1$ appears at each vertical segment of the outer border, as the definition suggests (see \cref{fig:abacus}). Focusing on entries $0$, we can equivalently define the Maya diagram using $\lambda'$:
	\[a_i = \begin{cases}
		0, & i=j-\lambda'_j \text{ for some }j;\\
		1, &\text{otherwise}.
	\end{cases}\]
	\ytableausetup{nosmalltableaux,centertableaux}
	\begin{figure}[htbp!]
		\centering
		\[\begin{array}{cccc}
			\begin{ytableau}
				0& 1 &2  &3 \\
				-1& 0 &1  & 2\\
				-2& -1 & 0\\
				-3
			\end{ytableau}&
			\begin{tikzpicture}[inner sep=0in,outer sep=0in, baseline=0pt]
				\def\0 {\raisebox{-5pt}{0}}
                \def\Csize {0.545cm} 
				\node (n) {
					\begin{ytableau}
						{}& {} &{}  &{} & \none[1\ \ \ ]\\
						{}& {} &{}  & \0 &\none[1\ \ \ ] \\
						{}& \0 & \raisebox{-5pt}{$\textcolor{red}{\textbf{0}}$} &\none[1\ \ \ ] \\
						\0 &\none[1\ \ \ ] \\
						\none[1\ \ \ ]\\
						\none[\vdots\ \ \ ]\\
				\end{ytableau}};
				\draw[red,thick] ([xshift=-\Csize] n.north east)--( n.north east) node [midway,above=1pt]{\textcolor{black}{0}};
				\draw[red,thick] (n.north east)--([xshift=\Csize] n.north east) node [midway,above=2pt]{\textcolor{black}{$\cdots$}};
				\draw[red,thick] (n.south west) 
				-- ([yshift=2*\Csize]n.south west)
				-- ([yshift=2*\Csize,xshift=\Csize]n.south west)
				-- ([yshift=3*\Csize,xshift=\Csize]n.south west)
				-- ([yshift=3*\Csize,xshift=3*\Csize]n.south west)
				-- ([yshift=4*\Csize,xshift=3*\Csize]n.south west)
				-- ([yshift=4*\Csize,xshift=4*\Csize]n.south west)
				-- ([yshift=6*\Csize,xshift=4*\Csize]n.south west);
			\end{tikzpicture} &
			
			\begin{array}{ccccc}
				0&1&2&3\\
				\hline\hline
				1 & 0 & \cdots & \\
				\textcolor{red}{\underline{0}} & 1 & 0 & 1\\
				1 & 0 & 1 & 0\\
				&& & \cdots
			\end{array}\\
			\text{contents}&        (\cdots,1,0,1,0,\textcolor{red}{\underline{0}},1,0,1,1,0,\cdots)&\text{$4$-abacus of $\lambda$}
		\end{array}\]
		\caption{The contents, Maya diagram and $4$-abacus of $\lambda=(4,4,3,1)$. The entry at content $0$ is highlighted.}
		\label{fig:abacus}
	\end{figure}
	
	In terms of the Maya diagram, adding/removing a ribbon of size $k$ corresponds to exchanging a pair of $0$ and $1$ with distance $k$, and the height of the ribbon is the number of $1$'s strictly between these two numbers.
	
	For a ribbon, we call its bottom-leftmost box the \emph{starting box}. We say the ribbon \emph{starts} at $i$ if the starting box has content $i$. In this language, removing a $k$-ribbon starting at $i$ means to exchange $0$ at $a_i$ and $1$ at $a_{i+k}$ in its Maya diagram.
	
	Given $k\in \N$, a $k$-\emph{abacus} of $\lambda$ is a cyclic arrangement of the Maya diagram $\{a_i\}$ in the form
	
	\[\begin{array}{ccccc}
		0&1&\cdots &k-2&k-1\\
		\hline\hline
		a_k & a_{k+1} & \cdots & a_{2k-2}& \cdots\\
		a_{0}&a_{1}&\cdots& a_{k-2}&a_{k-1}\\
		\cdots&a_{-k+1}&\cdots&a_{-2}&a_{-1}
	\end{array}.\]
    
	Note that the abacus should be seen cyclically, and there is no canonical meaning for two numbers on the same row. Nevertheless, we will always let $a_0$ be in the first column in this paper for consistency.

    In the abacus notation, removing a $k$-ribbon corresponds to exchanging $0$ and the $1$ above it. The uniqueness of $\core_k(\lambda)$ then follows --- it is the diagram obtained by moving all $0$'s above all $1$'s in each column of the $k$-abacus.
	
    \subsection{Symmetric functions}
	Let $\Lambda=\bigoplus_{n=0}^\infty \Lambda_n$ be the \emph{ring of symmetric functions} in infinite variables $(x_1,x_2,\dotsc)$ with $\Q$-coefficients, graded by the degree of the functions. This ring has several classic $\Q$-bases, among which we present a few:
	\begin{enumerate}[(1)]
		\item The \emph{monomial symmetric functions}: For each $\lambda\vdash n$, define $m_\lambda$ to be $\sum_{\alpha} x^\alpha$, where $\alpha$ runs through the rearrangements of parts of $\lambda$. 
		\item The \emph{elementary symmetric functions}: Define $e_k$ to be $\sum_{i_1<i_2<\cdots<i_k} x_{i_1}x_{i_2}\cdots x_{i_k}$, and for each $\lambda\vdash n$, define $e_\lambda= e_{\lambda_1}e_{\lambda_2}\cdots e_{\lambda_k}$.
		\item The \emph{homogeneous symmetric functions}: Define $h_k$ to be \\$\sum_{i_1\leq i_2\leq \cdots\leq i_k} x_{i_1}x_{i_2}\cdots x_{i_k}$, and for each $\lambda\vdash n$, define $h_\lambda= h_{\lambda_1}h_{\lambda_2}\cdots h_{\lambda_k}$.
            \item The \emph{power sum symmetric functions}: Define $p_k$ to be $\sum_{i}x_i^k$, and for each $\lambda\vdash n$, define $p_\lambda= p_{\lambda_1}p_{\lambda_2}\cdots p_{\lambda_k}$.
		\item The \emph{Schur symmetric functions}: For each $\lambda\vdash n$, define $s_\lambda$ by the Jacobi-Trudi formula $s_{\lambda}= \det\left(h_{\lambda_i-i+j}\right)_{1\leq i,j\leq \ell(\lambda)}$.
	\end{enumerate}
	By the definition of $e_{\lambda}$ and $h_{\lambda}$, saying that $\{e_\lambda:\lambda \in \Par\}$ and $\{h_\lambda:\lambda\in \Par\}$ are bases is saying that $\Lambda=\Q[e_1,e_2,\dotsc] = \Q[h_1,h_2,\dotsc]$.

	The \emph{skew Schur symmetric functions} are a generalized version of Schur functions, defined for each skew shape $\lambda/\mu$, as 
	\[s_{\lambda/\mu}= \det\left(h_{\lambda_i-i-\mu_j+j}\right)_{1\leq i,j\leq \ell(\lambda)}.\]
	
	There are many intricate relations between these symmetric functions. 
    Let $\mathbf{x}=(x_1,x_2,\dotsc)$ and $\mathbf{y}=(y_1,y_2,\dotsc)$. The RSK correspondence yields
	\[\sum_{\lambda\in \Par}s_\lambda(\mathbf{x})s_\lambda(\mathbf{y})t^{|\lambda|} = \prod_{i=1}^{\infty}\prod_{j=1}^\infty \frac{1}{1-x_iy_jt},\]
	or more generally, \cite[fourth equation on page 70]{macdonald1998symmetric}
	\[\sum_{\lambda\in \Par}s_{\lambda/\mu}(\mathbf{x})s_\lambda(\mathbf{y})t^{|\lambda|} = s_{\mu}(\mathbf{y})t^{|\mu|} \prod_{i=1}^{\infty}\prod_{j=1}^\infty \frac{1}{1-x_iy_jt}.\]
	
	On the other hand, the Cauchy identity provides another expansion formula
	\[\prod_{i=1}^{\infty}\prod_{j=1}^\infty \frac{1}{1-x_iy_jt} = \sum_{\lambda\in \Par} h_\lambda(\mathbf{x})m_{\lambda}(\mathbf{y})t^{|\lambda|}.\]

    There has been much research on multiplying a symmetric function by a Schur function. For example, Pieri rule tells that 
    \[h_n\cdot s_\mu = \sum_{\lambda} s_\lambda,\]
    where $\lambda$ runs through the partitions obtained by $\mu$ adding $n$ boxes with distinct column numbers. Dually, Pieri rule of $e_n$ tells that
    \[e_n\cdot s_\lambda = \sum_{\nu} s_\nu,\]
    where $\nu$ runs through the partitions obtained from $\lambda$ by adding $n$ boxes with distinct row numbers. 
    
    From the perspective of generating functions, we have
    \[E(-t)H(t):= \left(\sum_{n=0}^\infty e_n(\mathbf{x})(-t)^n\right)\left(\sum_{n=0}^\infty h_n(\mathbf{x})t^n\right)
    =\left(\prod_i (1-x_it)\right)\left(\prod_i \frac{1}{1-x_it}\right)=1,\]
    consequently,
    \[s_\mu = \left(\sum_{n=0}^\infty e_n(-t)^n\right)\left(\sum_{n=0}^\infty h_nt^n\right)s_\mu = \sum_\lambda\sum_\nu s_\nu (-1)^{|\nu|-|\lambda|},\]
    where $\lambda$ runs through partitions obtained from $\mu$ by adding any number of distinct-column boxes, and $\nu$ runs through $\lambda$ adding any number of distinct-row boxes.
    
    For a symmetric function $f(x_1,x_2,\dotsc)$, let $(f\circ p_k)= f(x_1^k,x_2^k,\dotsc)$, called the \emph{plethysm} of $f$ and $p_k$. For a more detailed definition of plethysm, see \cite[Sec. 8]{macdonald1998symmetric}.
        
        The generating function of $e_n\circ p_k$ is then \[E_k(t):=\sum_{n=0}^\infty (e_n\circ p_k)t^n = \prod_{i} (1+x_i^kt).\] 
        
        Note that 
        \[\sum_{n=0}^\infty G(k,n) t^n = \prod_{i} (1+x_it+\cdots+x_i^{k-1}t^{k-1}) = \prod_i \frac{1-x_i^kt^k}{1-x_it}= H(t)\cdot E_k(-t^k),\]
        we find $E_k(t)$ and $H(t)$ naturally appear in generating functions of Petrie symmetric functions.

        From this generating function decomposition, the Pieri rule, and the dual Plethystic Pieri rule (shown at \cref{thm:plethysticPieri}), one can obtain the so-called \textit{Petrie Pieri rule} in \cref{thm:refinedJJL} for Petrie symmetric functions.
        
        However, we will not use \cref{thm:plethysticPieri} throughout the paper, as one of our goals is to give an alternative proof of this theorem.

\section{Weighted sum formula of Petrie determinants }\label{sec:WeightSum}
    
    Let $P$ be a $n\times n$ Petrie matrix and $v[a_i,b_i]$, $1\leq i\leq n$ be its rows. Recall our definition of good orientations:

    \begin{definition}[Good orientation]
        An \emph{orientation} of a graph is an assignment of a direction to each edge in the graph, (or a reassignment if the graph itself is directed). For Petrie matrices $P$, an \textit{orientation} of $P$ is an (ordered) set of vectors $v[c_i,d_i]$, where either $(c_i,d_i)=(a_i,b_i)$ or $(c_i,d_i)=(b_i,a_i)$ for each $1\leq i\leq n$.

        An orientation $\sigma$ of $P$ is \emph{good} if $c_1,c_2,\dotsc,c_n$ is a permutation of $0,1,2,\dotsc,n-1$. Namely, in the Petrie graph, every vertex except $n$ has out-degree $1$ in the orientation. Good orientations of $P$ corresponds to good orientations of the Petrie graph of $P$.

        For a good orientation $\sigma$ with corresponding permutation $\tau=c_1c_2\cdots c_{n}$ (one-line notation that maps $i$ to $c_{i+1}$ for $0\leq i\leq n-1$), we denote its \emph{order} $|\sigma|$ to be the number of vectors whose signs are changed. We also assign its \emph{weight} $w(\sigma)$ to be $(-1)^{|\sigma|+\inv(\tau)}$.

    \end{definition}
        We can also act $\tau_\sigma$ to $\sigma$ by permuting the position of vectors in $\sigma$. This sorts the vectors by their first index.
    
        \begin{example}
        \cref{fig:PetrieGraphCycle} shows a Petrie matrix and one of its good orientations 
        \[\sigma= (v[4,0],v[2,4],v[3,6],v[1,5],v[0,1],v[5,2]),\]
        whose vectors form the middle matrix and the corresponding Petrie graph. 
        Its corresponding permutation $\tau_\sigma$ is $\tau_0\tau_1\tau_2\tau_3\tau_4\tau_5=423105$, which has $9$ inversions. 
        
        By acting $\tau_\sigma$ to $\sigma$ itself, we get
        \[\tau_\sigma\cdot \sigma = (v[0,1],v[1,5],v[2,4],v[3,6],v[4,0],v[5,2]).\]
        Its corresponding matrix and Petrie graph is shown in the last column of the figure. 
        The order of $\sigma$ is $2$, and its weight is $(-1)^2(-1)^9=-1$. 
        
        one can also check that the other good orientation of this Petrie matrix has weight $1$.
        \begin{figure}[hbtp!]
		\centering
		\[
        \begin{array}{ccccc}
		\begin{array}{c}
		  a\\b\\c\\d\\e\\f
		\end{array}\overset{\text{\scriptsize 0\quad 1\quad 2\quad 3\quad 4\quad  5\quad  6}}{\left[
			\begin{array}{C{1em}C{1em}C{1em}C{1em}C{1em}C{1em}}
				1&1&1&1&0&0\\
				0&0&1&1&0&0\\
				0&0&0&1&1&1\\
				0&1&1&1&1&0\\
				1&0&0&0&0&0\\
				0&0&1&1&1&0
		\end{array}\right]}
        &
		\implies
        &\begin{array}{c}
		  a\\b\\c\\d\\e\\f
		\end{array}
		\overset{\text{\scriptsize 0\quad 1\quad 2\quad 3\quad 4\quad  5\quad  6}}{\left[
			\begin{array}{C{1em}C{1em}C{1em}C{1em}C{1em}C{1em}}
				$-$&$-$&$-$&$-$&0&0\\
				0&0&$+$&$+$&0&0\\
				0&0&0&$+$&$+$&$+$\\
				0&$+$&$+$&$+$&$+$&0\\
				$+$&0&0&0&0&0\\
				0&0&$-$&$-$&$-$&0
		\end{array}\right]}
        &
		\implies
        &\begin{array}{c}
		  a\\b\\c\\d\\e\\f
		\end{array}
		\overset{\text{\scriptsize 0\quad 1\quad 2\quad 3\quad 4\quad  5\quad  6}}{\left[
			\begin{array}{C{1em}C{1em}C{1em}C{1em}C{1em}C{1em}}
				$+$&0&0&0&0&0\\
				0&$+$&$+$&$+$&$+$&0\\
				0&0&$+$&$+$&0&0\\
				0&0&0&$+$&$+$&$+$\\
				$-$&$-$&$-$&$-$&0&0\\
				0&0&$-$&$-$&$-$&0
		\end{array}\right]}
        \\
        P &&\text{matrix of }\sigma &&Q_\sigma\\
        \begin{tikzpicture}[baseline=-6pt]
			\draw (0:1.7) circle (7pt) node (a0){$0$};
			\draw (360*1/7:1.7) circle (7pt) node (a1){$1$};
			\draw (360*2/7:1.7) circle (7pt) node (a2){$2$};
			\draw (360*3/7:1.7) circle (7pt) node (a3){$3$};
			\draw (360*4/7:1.7) circle (7pt) node (a4){$4$};
			\draw (360*5/7:1.7) circle (7pt) node (a5){$5$};
			\draw (360*6/7:1.7) circle (7pt) node (a6){$6$};
			\draw[->] (a0) -- (a4) node[midway]{$a$};
			\draw[->] (a2) -- (a4) node[pos=0.4]{$b$};
			\draw[->] (a3) -- (a6) node[midway]{$c$};
			\draw[->] (a1) -- (a5) node[pos=0.4]{$d$};
			\draw[->] (a0) -- (a1) node[midway] {$e$};
			\draw[->] (a2) -- (a5) node[pos=0.4]{$f$};
		\end{tikzpicture}
		&\implies &
		\begin{tikzpicture}[baseline=-6pt]
			\draw (0:1.7) circle (7pt) node (a0){$0$};
			\draw (360*1/7:1.7) circle (7pt) node (a1){$1$};
			\draw (360*2/7:1.7) circle (7pt) node (a2){$2$};
			\draw (360*3/7:1.7) circle (7pt) node (a3){$3$};
			\draw (360*4/7:1.7) circle (7pt) node (a4){$4$};
			\draw (360*5/7:1.7) circle (7pt) node (a5){$5$};
			\draw (360*6/7:1.7) circle (7pt) node (a6){$6$};
			\draw[->,red] (a4) -- (a0) node[midway]{$a$};
			\draw[->] (a2) -- (a4) node[pos=0.4]{$b$};
			\draw[->] (a3) -- (a6) node[midway]{$c$};
			\draw[->] (a1) -- (a5) node[pos=0.4]{$d$};
			\draw[->] (a0) -- (a1) node[midway] {$e$};
			\draw[->,red] (a5) -- (a2) node[pos=0.6]{$f$};
		\end{tikzpicture}
		&\implies &
		\begin{tikzpicture}[baseline=-6pt]
			\draw (0:1.7) circle (7pt) node (a0){$0$};
			\draw (360*1/7:1.7) circle (7pt) node (a1){$1$};
			\draw (360*2/7:1.7) circle (7pt) node (a2){$2$};
			\draw (360*3/7:1.7) circle (7pt) node (a3){$3$};
			\draw (360*4/7:1.7) circle (7pt) node (a4){$4$};
			\draw (360*5/7:1.7) circle (7pt) node (a5){$5$};
			\draw (360*6/7:1.7) circle (7pt) node (a6){$6$};
			\draw[->] (a0) -- (a1) node[midway] {$a$};
			\draw[->] (a1) -- (a5) node[pos=0.4]{$b$};
			\draw[->] (a2) -- (a4) node[pos=0.4]{$c$};
			\draw[->] (a3) -- (a6) node[midway]{$d$};
			\draw[->,red] (a4) -- (a0) node[midway]{$e$};
			\draw[->,red] (a5) -- (a2) node[pos=0.6]{$f$};
		\end{tikzpicture}
        \end{array}\]
		\caption{A Petrie matrix, good orientation and its corresponding $Q_\sigma$ in the proof of \cref{thm:PetWeightInvo}. In the matrices, $+$ represents $+1$, and $-$ represents $-1$.}
		\label{fig:PetrieGraphCycle}
	\end{figure}
        \end{example}
        
        Good orientations on the graph can be characterized, whose proof is easy and is omitted:
    \begin{lemma}\label{lemma:graph}
        Let $G$ be a finite pseudograph\footnote{A pesudograph is a graph which allows self-loops and multiedges.} of $n$ vertices and $n-1$ edges. Let $r$ be a fixed vertex of $G$. Let $t_r(G)$ be the number of orientations of $G$ such that all vertices other than $r$ have exactly one out-degree. Then $t_r(G)\neq 0$ if and only if the component of $r$ is a tree and other components have exactly one cycle.\footnote{Such a component is called a \emph{unicyclic graph}.} 

        In this case, every good orientation has edges pointing directed to $r$ or the cycle in the component, and cycles are directed. As a consequence, $t_r(G)=2^{c(G)-1}$, where $c(G)$ is the number of components of $G$.
    \end{lemma}
		Note that \cref{lemma:graph} holds even when our graphs have self-loops or multiedges in the graph. In such case, we take the convention that a self-loop forms a $1$-cycle, double edges form a $2$-cycle, and each of them has two possible orientations.
	
        When there are $2^n$ good orientations, we have $n$ cycles in the graph, each induces an involution on the set of good orientations by reversing orientation of the cycle. We show that each involution is a weight-reversing bijection on the set of good orientations, yielding a counting formula for the determinant of the Petrie matrix.
        \begin{theorem}\label{thm:PetWeightInvo}
            Let $P$ be a Petrie matrix. 

            \begin{enumerate}
                \item If there is a unique good orientation $\sigma$ of $P$, then $\det(P)=w(\sigma)$.
                \item If there are more than one good orientations, let $\varphi$ be an involution induced by a cycle, then $w(\sigma)=-w(\varphi(\sigma))$ for any good orientation $\sigma$ of $P$.
            \end{enumerate}
            In either case,
            \[\det(P)=\sum_{\sigma:\text{good orientation}} w(\sigma).\]
        \end{theorem}
        \begin{proof}
            For each good orientation $\sigma$ with row vectors $v[c_1,d_1],\dotsc,v[c_n,d_n]$ and permutation $\tau$, consider $Q_\sigma$ to be the matrix of $\tau_\sigma\cdot \sigma$, i.e., the $n\times n$ matrix with row vectors $v[c_{\tau(i)},d_{\tau(i)}]$ for $0\leq i\leq n-1$. Then, $Q_\sigma$ is obtained from $P$ by negating $|\sigma|$ rows and permuting the rows by $\tau$. As a result, $\det(P)=w(\sigma)\det(Q_\sigma)$.
            
            If $\mathcal{P}(P)$ is a tree, then there is a unique good orientation $\sigma$. We show that $Q_\sigma$ has determinant $1$.

            We inductively attach the edge of a leaf to the root $n$. For a leaf $\ell$ whose edge is not attached to $n$, there is a directed path from $\ell$ to $n$, say $\ell\to u_1\to u_2\to \cdots\to i_k=n$. Hence, we apply the row operation adding $v[u_{i},u_{i+1}]$ from other rows to row $i$ ($=v[i,u_1]$) and get $v[i,n]$ without changing the determinant. When every vertex is attached to $n$, the resulting matrix $Q$ becomes an upper triangular matrix that has ones on the whole diagonal and upper triangular entries. Hence $\det(Q)=1$ and thus $\det(P)=w(\sigma)$.

            Next, suppose there is a good orientation $\sigma$ with permutation $\tau_\sigma$ and $C$ be a directed cycle in the Petrie graph induced by vectors of $\sigma$. 
            Let $\tau_C$ be the permutation of the cycle and $\varphi$ be the involution reversing the cycle $C$.
            
            Then, $\varphi(\sigma)$ is the good orientation obtained from $\sigma$ by switching $v[c_i,d_i]$ to $v[d_i,c_i]$ if $(c_i,d_i)$ is an edge of $C$. By the definition, we have $\tau_{\varphi(\sigma)} = \tau_C\cdot \tau_\sigma$ and $(-1)^{|\varphi(\sigma)|}=(-1)^{|C|+|\sigma|}$, where $|C|$ is the number of edges in $C$. As $\mathrm{sgn}(\tau_C)=(-1)^{|C|-1}$,
            \[(-1)^{|\varphi(\sigma)|}\mathrm{sgn}(\tau_{\varphi(\sigma)}) = (-1)^{|\sigma|} (-1)^{|C|} \mathrm{sgn}(\tau_C)\mathrm{sgn}(\tau_\sigma)= -(-1)^{|\sigma|}\mathrm{sgn}(\tau_\sigma),\]
            
        \end{proof}
        \begin{figure}[hbtp!]
		\centering
		\[
        \begin{array}{c}
		  a\\b\\c\\d\\e\\f
		\end{array}
		\overset{\text{\scriptsize 0\quad 1\quad 2\quad 3\quad 4\quad  5\quad  6}}{\left[
			\begin{array}{C{1em}C{1em}C{1em}C{1em}C{1em}C{1em}}
				$+$&0&0&0&0&0\\
				0&$+$&$+$&$+$&$+$&0\\
				0&0&$+$&$+$&0&0\\
				\textcolor{gray!50}{0}&\textcolor{gray!50}{0}&\textcolor{gray!50}{0}&\textcolor{gray!50}{$+$}&\textcolor{gray!50}{$+$}&\textcolor{gray!50}{$+$}\\
				$-$&$-$&$-$&$-$&0&0\\
				0&0&$-$&$-$&$-$&0
		\end{array}\right]}
		\implies
        \begin{array}{c}
		  a\\b\\c\\d\\e\\f
		\end{array}
		\overset{\text{\scriptsize 0\quad 1\quad 2\quad 3\quad 4\quad  5\quad  6}}{\left[
			\begin{array}{C{1em}C{1em}C{1em}C{1em}C{1em}C{1em}}
				$-$&0&0&0&0&0\\
				0&$-$&$-$&$-$&$-$&0\\
				0&0&$-$&$-$&0&0\\
                \textcolor{gray!50}{0}&\textcolor{gray!50}{0}&\textcolor{gray!50}{0}&\textcolor{gray!50}{$+$}&\textcolor{gray!50}{$+$}&\textcolor{gray!50}{$+$}\\
				$+$&$+$&$+$&$+$&0&0\\
				0&0&$+$&$+$&$+$&0
		\end{array}\right]}
		\implies
        \begin{array}{c}
		  a\\b\\c\\d\\e\\f
		\end{array}
		\overset{\text{\scriptsize 0\quad 1\quad 2\quad 3\quad 4\quad  5\quad  6}}{\left[
			\begin{array}{C{1em}C{1em}C{1em}C{1em}C{1em}C{1em}}
				$+$&$+$&$+$&$+$&0&0\\
				$-$&0&0&0&0&0\\
				0&0&$+$&$+$&$+$&0\\
				\textcolor{gray!50}{0}&\textcolor{gray!50}{0}&\textcolor{gray!50}{0}&\textcolor{gray!50}{$+$}&\textcolor{gray!50}{$+$}&\textcolor{gray!50}{$+$}\\
				0&0&$-$&$-$&0&0\\
				0&$-$&$-$&$-$&$-$&0
		\end{array}\right]}
		\]
		\[
		\begin{tikzpicture}[baseline=-6pt]
			\draw (0:1.7) circle (7pt) node (a0){$0$};
			\draw (360*1/7:1.7) circle (7pt) node (a1){$1$};
			\draw (360*2/7:1.7) circle (7pt) node (a5){$5$};
			\draw (360*3/7:1.7) circle (7pt) node (a2){$2$};
			\draw (360*4/7:1.7) circle (7pt) node (a4){$4$};
			\draw (360*5/7:1.7) circle (7pt) node (a3){$3$};
			\draw (360*6/7:1.7) circle (7pt) node (a6){$6$};
			\draw[->,red] (a4) -- (a0) node[midway,above]{$e$};
			\draw[->]     (a2) -- (a4) node[midway,right]{$c$};
			\draw[->,gray](a3) -- (a6) node[midway,above]{$d$};
			\draw[->]     (a1) -- (a5) node[midway,below]{$b$};
			\draw[->]     (a0) -- (a1) node[midway,left] {$a$};
			\draw[->,red] (a5) -- (a2) node[midway,below right]{$f$};
		\end{tikzpicture}
            \implies
		\begin{tikzpicture}[baseline=-6pt]
			\draw (0:1.7) circle (7pt) node (a0){$0$};
			\draw (360*1/7:1.7) circle (7pt) node (a1){$1$};
			\draw (360*2/7:1.7) circle (7pt) node (a5){$5$};
			\draw (360*3/7:1.7) circle (7pt) node (a2){$2$};
			\draw (360*4/7:1.7) circle (7pt) node (a4){$4$};
			\draw (360*5/7:1.7) circle (7pt) node (a3){$3$};
			\draw (360*6/7:1.7) circle (7pt) node (a6){$6$};
			\draw[->,red] (a1) -- (a0) node[midway,left] {$a$};
			\draw[->,red] (a5) -- (a1) node[midway,below]{$b$};
			\draw[->,red] (a4) -- (a2) node[midway,right]{$c$};
			\draw[->,gray](a3) -- (a6) node[midway]{$d$};
			\draw[->]     (a0) -- (a4) node[midway,above]{$e$};
			\draw[->]     (a2) -- (a5) node[midway,below]{$f$};
		\end{tikzpicture}
        \implies
        \begin{tikzpicture}[baseline=-6pt]
			\draw (0:1.7) circle (7pt) node (a0){$0$};
			\draw (360*1/7:1.7) circle (7pt) node (a1){$1$};
			\draw (360*2/7:1.7) circle (7pt) node (a5){$5$};
			\draw (360*3/7:1.7) circle (7pt) node (a2){$2$};
			\draw (360*4/7:1.7) circle (7pt) node (a4){$4$};
			\draw (360*5/7:1.7) circle (7pt) node (a3){$3$};
			\draw (360*6/7:1.7) circle (7pt) node (a6){$6$};
			\draw[->,red] (a1) -- (a0) node[midway,left] {$b$};
			\draw[->,red] (a5) -- (a1) node[midway,below]{$f$};
			\draw[->,red] (a4) -- (a2) node[midway,right]{$e$};
			\draw[->,gray] (a3) -- (a6) node[midway]{$d$};
			\draw[->] (a0) -- (a4) node[midway,above]{$a$};
			\draw[->] (a2) -- (a5) node[midway,below]{$c$};
		\end{tikzpicture}\]
		\caption{From one $Q_\sigma$ to $Q_{\varphi(\sigma)}$, we revert all edges on the cycle $(01524)$, and then apply the cyclic permutation on the rows of the matrix. Edges that point to smaller numbers are highlighted.}
		\label{fig:ReversePetrieCycle}
	\end{figure}
        \begin{example}
            \cref{fig:ReversePetrieCycle} shows two possible good orientations of
            \[\sigma= (v[4,0],v[2,4],v[3,6],v[1,5],v[0,1],v[5,2]).\]
            One can reach another by reversing the cycle of edges $0,1,5,2,4$ on the graph, which has row numbers $5,3,6,2,1$. The starting points of two $\tau$s differ by a cycle permutation $(01524)$. Hence, their weights differ by a factor of $(-1)^5\cdot (-1)^{5-1}=-1$.
        \end{example}

\section{\texorpdfstring{Petrie Graphs of $\Pet_k(\lambda,\mu)$}{Petrie Graphs of Pet(λ,μ)}}\label{sec:Graph}
	
	In this section, we introduce the Petrie graphs of $\Pet_k(\lambda,\mu)$ and show that this graph can be read directly from $k$-abaci of $\lambda$ and $\mu$. We first focus on the case $\mu = \varnothing$ as an introductory example.
    
    If $\mu=\varnothing$, every row in the Petrie matrix has $k$ $1$'s if the ones do not start/end at the border. For this reason, we use any label $\leq 0$ to refer to the first vertex $0$, and any label $\geq \ell(\lambda)$ for the last vertex $\ell(\lambda)$. Under this convention, an edge always connects some vertices of labels $i$ and $i+k$. By checking the indices, we have the following proposition:
    \begin{proposition} \label{prop:corr1}
		Let $\lambda$ be a partition, $k\in \N$. Let $\{a_j\}$ be the Maya diagram of $\lambda$. The data of Petrie graph $\mathcal{P}(\Pet_k(\lambda,\varnothing))$ can be read from the Maya diagram:
		\begin{enumerate}
			\item For $1-k\leq j< \ell(\lambda)$, $a_{-j}=1$ if and only if there is an edge from vertex $j$ to $j+k$.
			\item $\lambda_1\geq k$ if and only if there is a self-loop at the first vertex $0$.
		\end{enumerate}
	\end{proposition} 
	\begin{proof}\mbox{}
		\begin{enumerate}
			\item Recall that $(\Pet_k(\lambda,\varnothing))_{ij}=1$ if $0\leq \lambda_i-i+j<k$, i.e. $j\in [i-\lambda_i, k-1+i-\lambda_i]$. Hence, the $i$-th row of $\Pet_k(\lambda)$ is $v[i-\lambda_i-1,k-1+i-\lambda_i]$ under the prescribed convention.
			
			In the Maya diagram, $a_{-j}=1$ if and only if there is $i$ such that $-j=\lambda_i-i+1$, i.e. $j=i-\lambda_i-1$. The result follows.
			\item There is a loop at the first vertex if and only if $a_{j}=1$ for some $j\geq k$. In this case, we assume $j$ to be the largest index with $a_j=1$. Then we have $j=\lambda_1-1+1$, which means $\lambda_1\geq j\geq k$. The converse is also clear.
		\end{enumerate}
	\end{proof}
	\ytableausetup{smalltableaux}
	\begin{figure}[htbp!]
		\centering
		\[\begin{array}{cccc}
			\ydiagram{4,4,4,3,2,2,2,1,1} 
			& {}\hspace{1em}\scalebox{0.75}{$\begin{bmatrix}
					1& & & & & & & & \\
					1&1& & & & & & & \\
					1&1&1& & & & & & \\
					1&1&1&1&1& & & & \\
					& &1&1&1&1&1& & \\
					& & &1&1&1&1&1& \\
					& & & &1&1&1&1&1\\
					& & & & & &1&1&1\\
					& & & & & & &1&1
				\end{bmatrix}$}\hspace{1em}{}
			&
			\begin{tikzpicture}[baseline=-.73cm, scale=0.9,x=.65cm,y=0.83cm]
				\draw (0.5,0)-- (5.5,0) node [midway,above]{$\leq 0$};
				\draw (0.5,-3) -- (0.9,-3)--(0.9,-2)--(1.1,-2)--(1.1,-3)--(1.9,-3)--(1.9,-2)--(2.1,-2)--(2.1,-3)--(5.5, -3);
				\node at (3,-3) [below]{$\geq 9$};
				\draw (5,-1) circle (7pt) node[inner sep=3pt] (1){$1$};
				\draw (4,-1) circle (7pt) node[inner sep=3pt] (2){$2$};
				\draw (3,-1) circle (7pt) node[inner sep=3pt] (3){$3$};
				\draw (2,-1) circle (7pt) node[inner sep=3pt] (4){$4$};
				\draw (1,-1) circle (7pt) node[inner sep=3pt] (5){$5$};
				\draw (5,-2) circle (7pt) node[inner sep=3pt] (6){$6$};
				\draw (4,-2) circle (7pt) node[inner sep=3pt] (7){$7$};
				\draw (3,-2) circle (7pt) node[inner sep=3pt] (8){$8$};
				\draw[->] (5,0) -- (1);
				\draw[->] (4,0) -- (2);
				\draw[->] (3,0) -- (3);
				\draw[->] (1,0) -- (5);
				\draw[->] (2) -- (7);
				\draw[->] (3) -- (8);
				\draw[->] (4) -- (2,-2);
				\draw[->] (6) -- (5,-3);
				\draw[->] (7) -- (4,-3);
			\end{tikzpicture}
			& 
            \hspace{10pt}
			\renewcommand{\arraystretch}{1.8}
			\begin{array}{ccccc}
				0&1&2&3&4\\
				\hline\hline
				\underline{1} & 0 & 1 & 1&1\\
				0 & 1 & 1 & 1& 0\\
				\cdots&1& 0 & 1&1\\\hphantom{\cdots}&\hphantom{\cdots}&\hphantom{\cdots}&\hphantom{\cdots}&\hphantom{\cdots}
			\end{array}
			\renewcommand{\arraystretch}{1}
			\\
			\lambda & \Pet_5(\lambda,\varnothing) & \mathcal{P}(\Pet_5(\lambda,\varnothing)) &\hspace{8pt}\text{$5$-abacus of $\lambda$}
		\end{array}\]
		\caption{The diagram, $\Pet_5(\lambda,\varnothing)$, $\mathcal{P}(\Pet_5(\lambda,\varnothing))$ and $5$-abacus for $\lambda=(444322211)$.}
		\label{fig:simpleCase}
	\end{figure}
	\begin{example}
		Let $k=5$ and $\lambda= 444322211$. The corresponding objects for $\lambda$ are shown in \cref{fig:simpleCase}. 
		The top vertex $\leq 0$ and the bottom vertex $\geq 9$ are flattened, for better visualization.
		\cref{prop:corr1} tells that the pattern of vertical edges in $\mathcal{P}_5(\lambda,\varnothing)$ perfectly matches the pattern of ones in the $5$-abacus of $\lambda$. Here, $\mathcal{P}_5(\lambda,\varnothing)$ is a tree, and the congruence class $x\equiv 2\pmod 5$ links the top vertex and the bottom vertex. Indeed, $\pet_5(\lambda)=-1\neq 0$.
	\end{example}
    
	For general $\lambda,\mu$, the idea is to see $\Pet_k(\lambda,\mu)$ as a submatrix of an $\ell(\lambda) \times (\ell(\lambda)+\mu_1)$ matrix
	\[M_{\lambda,\mu}:=([0\leq \lambda_i-i+j < k])_{\substack{1\leq i\leq \ell(\lambda),\\ -\mu_1+1\leq j\leq \ell(\lambda)}},\]
	from which we obtain $\Pet_k(\lambda,\mu)$ by removing columns whose indices are not of the form $-\mu_j+j$; see \cref{fig:PetSkew}. Here, we label each column by the value $j$ instead of natural numbers to best align with the case $\mu=\varnothing$. 
	
	Removing column $i$ from a Petrie matrix corresponds to merging $i$ and $i-1$ in its Petrie graph. Instead of doing the mergings precisely, we add horizontal edges to mark the merging, and call this new graph $\mathcal{P}'(\Pet_k(\lambda,\mu))$. $\mathcal{P}(\Pet_k(\lambda,\mu))$ is then obtained from $\mathcal{P}'(\Pet_k(\lambda,\mu))$ by contracting these horizontal edges.
    
    \ytableausetup{smalltableaux}
    \setcounter{MaxMatrixCols}{15}
	\begin{figure}[htbp!]
		\centering
		
		\[\begin{array}{ccc}
			\ydiagram{4+2,2+4,2+4,2+4,2+3,1+4,1+3,1+3,2} 
			& \scalebox{0.75}{$\overset{\begin{tikzpicture}[x=12.5pt,xshift=-8pt]
            \node at (-12.1,0) {-3}; 
            \node at (-11.1,0) {-2}; 
            \node at (-10.1,0) {-1}; 
            \node at (-9,0) { 0}; 
            \node at (-8,0) { 1}; 
            \node at (-7,0) { 2}; 
            \node at (-6,0) { 3}; 
            \node at (-5,0) { 4}; 
            \node at (-4,0) { 5}; 
            \node at (-3,0) { 6}; 
            \node at (-2,0) { 7}; 
            \node at (-1,0) { 8}; 
            \node at ( 0,0) { 9}; 
            \end{tikzpicture}}{\begin{bmatrix}
						\,1\,&\,\textcolor{red}{1}\,&\,\textcolor{red}{1}\,& & & & & & & & \\
						\,1\,&\,\textcolor{red}{1}\,&\,\textcolor{red}{1}\,&\,1\,& & & & & & & \\
						\,1\,&\,\textcolor{red}{1}\,&\,\textcolor{red}{1}\,&\,1\,&\,1\,& & & & \\
						&\,\textcolor{red}{1}\,&\,\textcolor{red}{1}\,&\,1\,&\,1\,&\,1\,& & & \\
						& & &\,1\,&\,1\,&\,1\,&\,1\,&\,\textcolor{red}{1}\,& \\
						& & & &\,1\,&\,1\,&\,1\,&\,\textcolor{red}{1}\,&\,1\,\\
						& & & & & &\,1\,&\,\textcolor{red}{1}\,&\,1\,&\,1\,&\,1\,\\
						& & & & & & &\,\textcolor{red}{1}\,&\,1\,&\,1\,&\,1\,&\,\textcolor{red}{1}\,&\\
						& & & & & & & & & &\,1\,&\,\textcolor{red}{1}\,&\,1\,\\
				\end{bmatrix}}$}
			& \scalebox{0.75}{$\overset{\begin{tikzpicture}[scale=.9, x=13.3pt,y=8pt, transform shape]
            \node at (-13.5,0) {-4}; 
            \node at (-12.3,0) {-3}; 
            \node at (-11.7,1) {-2}; 
            \node at (-11.1,0) {-1}; 
            \node at (-9.8,0) { 0}; 
            \node at (-8.8,0) { 1}; 
            \node at (-7.8,0) { 2}; 
            \node at (-6.65,0) { 3}; 
            \node at (-6.25,0) { 4}; 
            \node at (-5.1,0) { 5}; 
            \node at (-4.1,0) { 6}; 
            \node at (-2.95,0) { 7}; 
            \node at (-2.55,0) { 8}; 
            \node at (-1.4,0) { 9}; 
						\draw[color=white] (-1,0) circle (2pt); 
            \end{tikzpicture}}{\begin{bmatrix}
						\,1\,\quad~& & & & & & & & \\
						\,1\,\quad~&~\,1\,& & & & & & & \\
						\,1\,\quad~&~\,1\,&\,1\,& & & & \\
						&~\,1\,&\,1\,&\,1\,& & & \\
						&~\,1\,&\,1\,&\,1\,&\,1\,~& \\
						&  &\,1\,&\,1\,&\,1\,~&~\,1\,\\
						&  & & &\,1\,~&~\,1\,&\,1\,&\,1\,~\\
						&  & & &      &~\,1\,&\,1\,&\,1\,~&~\,1\,\\
						&  & & &      &      &     &\,1\,~&~\,1\,\\
				\end{bmatrix}}$}
			\\
			{\lambda/\mu} & M_{\lambda,\mu} \text{ and columns to be removed}
			&
			\Pet_5(\lambda,\mu)
			\\
			\begin{tikzpicture}[baseline=.39cm, scale=0.9,y=0.85cm]
				\draw [red] (1,.7)-- (4.5,.7) -- (4.5, -0.2) -- (5, -0.2);
				\node at (3,.6) [above]{$\leq -4$};
				
                \draw (1,-2.7) -- (1.3,-2.7)--(1.3,-1.7)--(1.5,-1.7)--(1.5,-2.7)--(2.1,-2.7)--(2.1,-1.7)--(2.3,-1.7)--(2.3,-2.7)--(5, -2.7);
				\node at (3,-2.6) [below]{$\geq 9$};
				
				\draw (3.8, 0) circle (6pt) node[inner sep=2pt] (-3){\scriptsize -3};
				\draw (3, 0) circle (6pt) node[inner sep=2pt] (-2){\scriptsize -2};
				\draw (2.2, 0) circle (6pt) node[inner sep=2pt] (-1){\scriptsize -1};
				\draw (1.4, 0) circle (6pt) node[inner sep=2pt] (0){$0$};
				\draw (4.6,-1) circle (6pt) node[inner sep=2pt] (1){$1$};
				\draw (3.8,-1) circle (6pt) node[inner sep=2pt] (2){$2$};
				\draw (3  ,-1) circle (6pt) node[inner sep=2pt] (3){$3$};
				\draw (2.2,-1) circle (6pt) node[inner sep=2pt] (4){$4$};
				\draw (1.4,-1) circle (6pt) node[inner sep=2pt] (5){$5$};
				\draw (4.6,-2) circle (6pt) node[inner sep=2pt] (6){$6$};
				\draw (3.8,-2) circle (6pt) node[inner sep=2pt] (7){$7$};
				\draw (3  ,-2) circle (6pt) node[inner sep=2pt] (8){$8$};
				\draw[red,-] (-3) -- (-2);
				\draw[red,-] (-2) -- (-1);
				\draw[red,-] (3) -- (4);
				\draw[red,-] (7) -- (8);
				\draw[->] (2.2, .7) -- (-1);
				\draw[->] (1.4, .7) -- (0);
				\draw[->] (4.6,-.2) -- (1);
				\draw[->] (-3) -- (2);
				\draw[->] (-1) -- (4);
				\draw[->] (0) -- (5);
				\draw[->] (2) -- (7);
				\draw[->] (3) -- (8);
				\draw[->] (6) -- (4.6,-2.7);
			\end{tikzpicture}
			& 
			\renewcommand{\arraystretch}{1.8}
			\begin{array}{ccccc}
				\hphantom{\cdots\cdot}&\hphantom{\cdots\cdot}&\hphantom{\cdots\cdot}&\hphantom{\cdots\cdot}&\hphantom{\cdots\cdot}\\
				0&1&2&3&4\\
				\hline\hline
				1 & 1 & 0 & 0 & \cdots \\
				\underline{1} & 1 & 0 & 1&1\\
				0 & 0 & 1 & 1& 0\\
				\cdots&1& 0 & 0&1
			\end{array}
			\renewcommand{\arraystretch}{1}
			&
			\renewcommand{\arraystretch}{1.8}
			\begin{array}{ccccc}
				\hphantom{\cdots\cdot}&\hphantom{\cdots\cdot}&\hphantom{\cdots\cdot}&\hphantom{\cdots\cdot}&\hphantom{\cdots\cdot}\\
				0&1&2&3&4\\
				\hline\hline
				\textcolor{red}{0} & \textcolor{red}{0} & \textcolor{red}{0} & \textcolor{red}{0} & \cdots \\
				\underline{1} & 1 & \textcolor{red}{0} & \textcolor{red}{0}&1\\
				1 & 1 & \textcolor{red}{0} & 1& 1\\
				\cdots&1& 1 & \textcolor{red}{0}&1
			\end{array}
			\renewcommand{\arraystretch}{1}
			\\
			\mathcal{P}'(\Pet_5(\nu,\mu)) & \text{$5$-abacus of $\lambda$} & \text{$5$-abacus of $\mu$}
		\end{array}\]
		\caption{The skew shape, $\Pet_5(\lambda,\mu)$, $\mathcal{P}'(\Pet_5(\nu,\mu))$ and $5$-abaci for $\lambda=666655442,\mu=42222111$. Horizontal edges are shown in red. Note that in the first matrix, we mark the column numbers, whereas in the second matrix, we mark the vertices of the Petrie graph.}
		\label{fig:PetSkew}
	\end{figure}
        
        To differentiate the vertices of $\mathcal{P}(\Pet_k(\lambda,\mu))$ and $\mathcal{P}'(\Pet_k(\lambda,\mu))$, we give them different names:
            
	\begin{definition}[cell/cluster]
            We call the vertex $\leq -\mu_1$ the \emph{top vertex} and $\geq \lambda_1$ the \emph{bottom vertex}.
            
            Fix partitions $\lambda,\mu$ and $k\in \N$. We call it a \emph{cell} for a non-bottom vertex in $\mathcal{P}'(\Pet_k(\lambda,\mu))$, and a \emph{cluster} for a non-bottom vertex in $\mathcal{P}(\Pet_k(\lambda,\mu))$.
            
            Given a cell $c$, the cluster of $c$ is the cluster that $c$ merges into.
	\end{definition}
         
        We use $\mathcal{P}'(\Pet_k(\lambda,\mu))$ more because it is easier to read and draw from the $k$-abaci of $\lambda$ and $\mu$. 
	\begin{proposition}\label{prop:PetrieAbacus}
		Given partitions $\lambda,\mu$ and $k\in \N$. Let $(a_i)_{i=-\infty}^\infty$, $(b_i)_{i=-\infty}^{\infty}$ be the Maya diagrams of $\lambda,\mu$, respectively. Then in $\mathcal{P}'(\Pet_k(\lambda,\mu))$,
		\begin{enumerate}
			\item For each $j< \ell(\lambda)$, the cells $j$ and $j+k$ are adjacent if and only if $a_{-j}=1$.
			\item For each $i< \ell(\lambda)$, the cells $i$ and $i+1$ are merged if and only if $b_{-i}=0$.
		\end{enumerate}
	\end{proposition}
	\begin{proof}
            For the first case, we check that the $i$-th row of $M_{\lambda,\mu}$ is $v[i-\lambda_i-1, i-\lambda_i-1+k]$ since $a_{-j}=1$ if $-j=\lambda_i-i+1$ for some $i$. Indeed, $0\leq \lambda_i-i+j\leq k-1$ iff $j\in [i-\lambda_i,\, i-\lambda_i-1+k]$.

            For the second one, as vertices $i$ and $i+1$ are merged if $i+1$ is not in the form $-\mu_j+j$, it means $-i$ is not of the form $\mu_j-j+1$, i.e. $b_{-i}=0$.
	\end{proof}

    The proposition suggests that the top vertex is in fact formed by gluing all vertices before $-\mu_1$, whereas the bottom vertex should be understood as $k$ vertically glued vertices that are glued together at the very end of the abacus. This also justifies why we exclude the bottom vertex from our definition of cells.
\section{Removing ribbons and Proof of \texorpdfstring{\cref{thm:refinedJJL}}{Theorem 1.4}}\label{sec:Ribbon}

	The fact that $\mathcal{P}'(\Pet_k(\lambda,\mu))$ can be read off from $k$-abaci indicates the close involvement of $k$-ribbon calculus. In this section, we elaborate on how removing/adding ribbon(s) on $\lambda/\mu$ corresponds to actions on Petrie graphs. Recall that we say a ribbon starts at $i$ if its starting box, the bottom-leftmost box, has content $i$.
	
	\begin{proposition} \label{prop:removingHook}
		Let $\lambda/\mu$ be a skew shape and $\gamma=\lambda/\nu$ be a $k$-ribbon of $\lambda/\mu$ with height $q$ which starts at $-i$. Then removing $\gamma$ corresponds to
		\begin{enumerate}[(i)]
			\item Young diagram: Changing $\lambda/\mu$ to $\nu/\mu$. The height $q$ is one less than the number of affected rows.
			\item Maya diagram: Changing $a_{-i}$ to $1$ and $a_{-i+k}$ to $0$. The height $q$ is the number of ones whose indices are strictly between $-i$ and $-i+k$.
			\item Abacus of $\lambda$: Exchanging the $0$ at the index $-i$ and the $1$ above it. 
			\item $\mathcal{P}(\Pet_k(\lambda,\mu))$: Moving the edge $(i-k,i)$ down to $(i,i+k)$. The height $q$ is the number of edges $(j,j+k)$ with $j$ between $i-k$ and $i$. 
			\item $\Pet_k(\lambda,\mu)$: Exchanging the row having $1$'s at indices $i-k+1$ to $i$ (in $M_{\lambda,\mu}$) to a row with $1$'s at indices $i+1$ to $i+k$, then move this row $q$ rows below.
		\end{enumerate}
	\end{proposition}
	\begin{proof}
		The correspondences of the first three are direct, and the correspondence between part (iii) and (iv) is translated by \cref{prop:PetrieAbacus}.
		
		For $\Pet_k(\lambda,\mu)$, the corresponding action is to change the row and move the row to the place to order the rows. The number of rows moved below is the number of edges with a starting point between $i-k$ and $i$, which is $q$ by part (iv).
    \end{proof}
    
    As $\Pet_k(\lambda,\mu)$ is a Petrie matrix and thus \cref{thm:PetWeightInvo} applies, here we translate the definition of the weight of a good orientation using the setting of $\Pet_k(\lambda,\mu)$.

    \begin{definition}[Basic graph, $k$-decomposition]
        Let $\mu\subseteq \nu\subseteq \lambda$ be partitions. 
        
        A Petrie graph $\mathcal{P}(\Pet_k(\nu,\mu))$ (or $\mathcal{P}'(\Pet_k(\nu,\mu))$) is  \emph{basic} if each cluster has exactly one downward edge. That is, the usual small-to-large orientation of the Petrie graph $\mathcal{P}(\Pet_k(\nu,\mu))$ is good.
		
        A \textit{$k$-decomposition} of $\lambda$ to $\nu$ is a sequence of partitions $\lambda_0=\nu \subset \lambda_1\subset \cdots\subset \lambda_n=\lambda$ where $\lambda_i/\lambda_{i-1}$ is a $k$-ribbon starting at $-c_i$ for $-c_1>-c_2>\cdots>-c_n$. We sometimes view $k$-decompositions as its induced ribbons.
        (In terms of Petrie graphs, they corresponds to moving the edges $(c_i-k,c_i)$ to $(c_i,c_i+k)$ one per step, from bottom to top.)
	\end{definition}
         To a good orientation $\sigma$ of $\mathcal{P}(\Pet_k(\lambda,\mu))$ we may associate a basic graph and a $k$-decomposition:
         \begin{enumerate}
             \item We may consider a graph $\mathcal{P}_\sigma$, defined to be the graph obtained from $\mathcal{P}'(\Pet_k(\lambda,\mu))$ by rotating each upward arrow $180^\circ$ with respect to its beginning point, so that it now points to the lower cluster instead of the upper cluster. In terms of unordered edges, we move down the upward edges by one cell.
             
             As $\sigma$ is a good orientation, $\mathcal{P}_\sigma$ is basic. In particular, we will not collide edges when rotating those upward edges if starting from bottom to top.
             \item By rotating edges from bottom to top, it induces a $k$-decomposition of $\lambda$ to some $\nu$. Note that $\mathcal{P}'(\Pet_k(\nu,\mu))=\mathcal{P}_\sigma$.
         \end{enumerate}
         
        The terms $k$-decomposition and $k$-decomposability also appear in other papers \cite{MR4881678, MR3537003} when stating the plethystic Murnaghan-Nakayama rule. We will also see later that requiring $-c_1>-c_2>\cdots>-c_n$ ensures that the ribbons are left-justified.
        \begin{figure}[htbp!]
		\centering
		\[\begin{array}{ccccc}
        
			\begin{tikzpicture}[baseline=-1cm, scale=0.9]
				\draw [red] (1,.7)-- (4.5,.7) -- (4.5, -0.2) -- (5, -0.2);
				\node at (3,.6) [above]{$\leq -4$};
				\draw (1,-2.7) -- (1.3,-2.7)--(1.3,-1.7)--(1.5,-1.7)--(1.5,-2.7)--(2.1,-2.7)--(2.1,-1.7)--(2.3,-1.7)--(2.3,-2.7)--(5, -2.7);
				\node at (3,-2.6) [below]{$\geq 9$};
				
				\draw (3.8, 0) circle (6pt) node[inner sep=2pt] (-3){\scriptsize -3};
				\draw (3, 0) circle (6pt) node[inner sep=2pt] (-2){\scriptsize -2};
				\draw (2.2, 0) circle (6pt) node[inner sep=2pt] (-1){\scriptsize -1};
				\draw (1.4, 0) circle (6pt) node[inner sep=2pt] (0){$0$};
				\draw (4.6,-1) circle (6pt) node[inner sep=2pt] (1){$1$};
				\draw (3.8,-1) circle (6pt) node[inner sep=2pt] (2){$2$};
				\draw (3  ,-1) circle (6pt) node[inner sep=2pt] (3){$3$};
				\draw (2.2,-1) circle (6pt) node[inner sep=2pt] (4){$4$};
				\draw (1.4,-1) circle (6pt) node[inner sep=2pt] (5){$5$};
				\draw (4.6,-2) circle (6pt) node[inner sep=2pt] (6){$6$};
				\draw (3.8,-2) circle (6pt) node[inner sep=2pt] (7){$7$};
				\draw (3  ,-2) circle (6pt) node[inner sep=2pt] (8){$8$};
				\draw[red,-] (-3) -- (-2);
				\draw[red,-] (-2) -- (-1);
				\draw[red,-] (3) -- (4);
				\draw[red,-] (7) -- (8);
				\draw[->] (2.2,.7) -- (-1);
				\draw[->] (-1) -- (4);
				\draw[->] (3) -- (8);
				\draw[->] (6) -- (4.6,-2.7);
				
				\draw[blue, <-] (1.4,.7) -- (0) node[midway,right]{1};
				\draw[blue, <-] (4.6,-.2) -- (1) node[midway,right]{2};
				\draw[blue, <-] (-3)-- (2) node[midway,right]{3};
				\draw[blue, <-] (0)-- (5) node[midway,right]{4};
				\draw[blue, <-] (2) -- (7) node[midway,right]{5};
			\end{tikzpicture}\quad 
			& \to&
			\begin{tikzpicture}[baseline=-1cm, scale=0.9]
				\draw [red] (1,.7)-- (4.5,.7) -- (4.5, -0.2) -- (5, -0.2);
				\node at (3,.6) [above]{$\leq -4$};
				\draw (1,-2.7) -- (1.3,-2.7)--(1.3,-1.7)--(1.5,-1.7)--(1.5,-2.7)--(2.1,-2.7)--(2.1,-1.7)--(2.3,-1.7)--(2.3,-2.7)--(5, -2.7);
				\node at (3,-2.6) [below]{$\geq 9$};
				
				\draw (3.8, 0) circle (6pt) node[inner sep=2pt] (-3){\scriptsize -3};
				\draw (3, 0) circle (6pt) node[inner sep=2pt] (-2){\scriptsize -2};
				\draw (2.2, 0) circle (6pt) node[inner sep=2pt] (-1){\scriptsize -1};
				\draw (1.4, 0) circle (6pt) node[inner sep=2pt] (0){$0$};
				\draw (4.6,-1) circle (6pt) node[inner sep=2pt] (1){$1$};
				\draw (3.8,-1) circle (6pt) node[inner sep=2pt] (2){$2$};
				\draw (3  ,-1) circle (6pt) node[inner sep=2pt] (3){$3$};
				\draw (2.2,-1) circle (6pt) node[inner sep=2pt] (4){$4$};
				\draw (1.4,-1) circle (6pt) node[inner sep=2pt] (5){$5$};
				\draw (4.6,-2) circle (6pt) node[inner sep=2pt] (6){$6$};
				\draw (3.8,-2) circle (6pt) node[inner sep=2pt] (7){$7$};
				\draw (3  ,-2) circle (6pt) node[inner sep=2pt] (8){$8$};
				\draw[red,-] (-3) -- (-2);
				\draw[red,-] (-2) -- (-1);
				\draw[red,-] (3) -- (4);
				\draw[red,-] (7) -- (8);
				\draw[->] (2.2,.7) -- (-1);
				\draw[->] (-1) -- (4);
				\draw[->] (3) -- (8);
				\draw[->] (6) -- (4.6,-2.7);
				
				\draw[blue, ->] (0) -- (5) node[midway,right]{1};
				\draw[blue, ->] (1) -- (6) node[midway,right]{2};
				\draw[blue, ->] (2) -- (7) node[midway,right]{3};
				\draw[blue, ->] (5) -- (1.4,-1.7) node[midway,right]{4};
				\draw[blue, ->] (7) -- (3.8,-2.7) node[midway,right]{5};
			\end{tikzpicture}\quad 
			& \to &
			\ytableaushort
			{
				\none\none\none\none {*(black)0}{*(black)0},
				\none\none1111,
				\none\none1222,
				\none\none2233,
				\none\none333,
				\none{*(black)0}444,
				\none445,
				\none555,
				{*(black)0}5}
			\\
			\text{A good orientation } \sigma_1 && \mathcal{P}_{\sigma_1} && \text{$k$-decomposition of $\sigma_1$}
            \\
			\ytableaushort
			{
				\none\none\none\none {*(black)0}{*(black)0},
				\none\none{*(black)0}{*(black)0}11,
				\none\none1112,
				\none\none2222,
				\none\none333,
				\none3344,
				\none444,
				\none555,
				55}
			\quad 
			& \to&
			\begin{tikzpicture}[baseline=-1cm, scale=0.9]
				\draw [red] (1,.7)-- (4.5,.7) -- (4.5, -0.2) -- (5, -0.2);
				\node at (3,.6) [above]{$\leq -4$};
				\draw (1,-2.7) -- (1.3,-2.7)--(1.3,-1.7)--(1.5,-1.7)--(1.5,-2.7)--(2.1,-2.7)--(2.1,-1.7)--(2.3,-1.7)--(2.3,-2.7)--(5, -2.7);
				\node at (3,-2.6) [below]{$\geq 9$};
				
				\draw (3.8, 0) circle (6pt) node[inner sep=2pt] (-3){\scriptsize -3};
				\draw (3, 0) circle (6pt) node[inner sep=2pt] (-2){\scriptsize -2};
				\draw (2.2, 0) circle (6pt) node[inner sep=2pt] (-1){\scriptsize -1};
				\draw (1.4, 0) circle (6pt) node[inner sep=2pt] (0){$0$};
				\draw (4.6,-1) circle (6pt) node[inner sep=2pt] (1){$1$};
				\draw (3.8,-1) circle (6pt) node[inner sep=2pt] (2){$2$};
				\draw (3  ,-1) circle (6pt) node[inner sep=2pt] (3){$3$};
				\draw (2.2,-1) circle (6pt) node[inner sep=2pt] (4){$4$};
				\draw (1.4,-1) circle (6pt) node[inner sep=2pt] (5){$5$};
				\draw (4.6,-2) circle (6pt) node[inner sep=2pt] (6){$6$};
				\draw (3.8,-2) circle (6pt) node[inner sep=2pt] (7){$7$};
				\draw (3  ,-2) circle (6pt) node[inner sep=2pt] (8){$8$};
				\draw[red,-] (-3) -- (-2);
				\draw[red,-] (-2) -- (-1);
				\draw[red,-] (3) -- (4);
				\draw[red,-] (7) -- (8);
				\draw[->] (2.2,.7) -- (-1);
				\draw[->] (6) -- (4.6,-2.7);
				\draw[->] (-3)-- (2);
				\draw[->] (2) -- (7);
				
				\draw[blue, ->] (0)--(5) node[midway,right]{1};
				\draw[blue, ->] (1)--(6) node[midway,right]{2};
				\draw[blue, ->] (4)--(2.2,-1.7) node[midway,right]{3};
				\draw[blue, ->] (5)--(1.4,-1.7) node[midway,right]{4};
				\draw[blue, ->] (8)--(3  ,-2.7) node[midway,right]{5};
			\end{tikzpicture}
			\quad 
			& \to&
			\begin{tikzpicture}[baseline=-1cm, scale=0.9]
				\draw [red] (1,.7)-- (4.5,.7) -- (4.5, -0.2) -- (5, -0.2);
				\node at (3,.6) [above]{$\leq -4$};
				\draw (1,-2.7) -- (1.3,-2.7)--(1.3,-1.7)--(1.5,-1.7)--(1.5,-2.7)--(2.1,-2.7)--(2.1,-1.7)--(2.3,-1.7)--(2.3,-2.7)--(5, -2.7);
				\node at (3,-2.6) [below]{$\geq 9$};
				
				\draw (3.8, 0) circle (6pt) node[inner sep=2pt] (-3){\scriptsize -3};
				\draw (3, 0) circle (6pt) node[inner sep=2pt] (-2){\scriptsize -2};
				\draw (2.2, 0) circle (6pt) node[inner sep=2pt] (-1){\scriptsize -1};
				\draw (1.4, 0) circle (6pt) node[inner sep=2pt] (0){$0$};
				\draw (4.6,-1) circle (6pt) node[inner sep=2pt] (1){$1$};
				\draw (3.8,-1) circle (6pt) node[inner sep=2pt] (2){$2$};
				\draw (3  ,-1) circle (6pt) node[inner sep=2pt] (3){$3$};
				\draw (2.2,-1) circle (6pt) node[inner sep=2pt] (4){$4$};
				\draw (1.4,-1) circle (6pt) node[inner sep=2pt] (5){$5$};
				\draw (4.6,-2) circle (6pt) node[inner sep=2pt] (6){$6$};
				\draw (3.8,-2) circle (6pt) node[inner sep=2pt] (7){$7$};
				\draw (3  ,-2) circle (6pt) node[inner sep=2pt] (8){$8$};
				\draw[red,-] (-3) -- (-2);
				\draw[red,-] (-2) -- (-1);
				\draw[red,-] (3) -- (4);
				\draw[red,-] (7) -- (8);
				\draw[->] (2.2,.7) -- (-1);
				\draw[->] (6) -- (4.6,-2.7);
				\draw[->] (-3)-- (2);
				\draw[->] (2) -- (7);
				
				\draw[blue, <-] (1.4,.7) -- (0) node[midway,right]{1};
				\draw[blue, <-] (4.6,-.2) -- (1) node[midway,right]{2};
				\draw[blue, <-] (-1) -- (4) node[midway,right]{3};
				\draw[blue, <-] (0)-- (5) node[midway,right]{4};
				\draw[blue, <-] (3) -- (8) node[midway,right]{5};
			\end{tikzpicture}
			\\
			\text{A good pair } (\nu,\Theta) && \mathcal{P}'(\Pet_5(\nu,\mu)) && \text{resulting orientation }\sigma_2\vspace{20pt}
		\end{array}\]
		\caption{Two examples of correspondence between good pairs and good orientations of $\lambda/\mu=666655442/42222111$}
		\label{fig:proper_correspondence}
	\end{figure}
    \begin{example}
        The first row of \cref{fig:proper_correspondence} shows a good orientation $\sigma_1$ of $\mathcal{P}'(\Pet_5(\lambda,\mu))$ with $\lambda/\mu=666655442/42222111$, its corresponding basic $\mathcal{P}_{\sigma_1}$, and the ribbons in the $k$-decomposition of $\sigma$. The black part represents $\nu/\mu$ and the numbers represents the ribbons in the $k$-decomposition of $\sigma_1$.
    \end{example}
        
    \begin{proposition}\label{prop:initial}
        Let $\mu\subseteq \nu\subseteq \lambda$ be partitions. 
        \begin{enumerate}
            \item Then $\nu/\mu$ is a horizontal strip if and only if $\mathcal{P}'(\Pet_k(\nu,\mu))$ is basic. 
            
            In this case, $\mathcal{P}'(\Pet_k(\nu,\mu))$ is a tree if and only if $\nu/\mu$ is a horizontal strip whose size of each part is smaller than $k$.
            \item Let $\sigma$ be a good orientation of $\mathcal{P}(\Pet_k(\lambda,\mu))$ (with the root being the bottom vertex). Then its order $|\sigma|$ equals the number of upward edges of $\sigma$, and $\inv(\tau_\sigma)$ is the sum of the height of the ribbons in the $k$-decomposition induced by $\sigma$.
        \end{enumerate}
    \end{proposition}
    \begin{proof}\ \\
    \begin{enumerate}
        \item 
        A cluster of $\mathcal{P}'(\Pet_k(\nu,\mu))$ consists of an interval of cells $[-\mu_j+j,-\mu_{j+1}+j]$, where we define $\mu_0=\infty$. On the other hand, $\nu/\mu$ is a horizontal strip if and only if their sequences are interlaced, i.e., $\nu_1\ge \mu_1\ge \nu_2\ge \cdots$. 

        If $\nu/\mu$ is a horizontal strip, then for each $j\geq 0$, there is only one $i$, namely $j+1$, that satisfies $-\nu_i+i-1 \in [-\mu_j+j,-\mu_{j+1}+j]$. Hence, only the cell $-\nu_{j+1}+j$ has an edge going down from this cluster.

        Conversely, if $\mathcal{P}'(\Pet_k(\nu,\mu))$ is basic, then for each $j\geq 0$, there is a unique $i$ such that $-\nu_i+i-1 \in [-\mu_j+j,-\mu_{j+1}+j]$. This forces $i=j+1$. By plugging in $i=j+1$ we obtain $\mu_j\geq \nu_{j+1}$ for any $j\geq 0$.

        Now assume that $\mathcal{P}'(\Pet_k(\lambda,\mu))$ is basic. Notice that $\mathcal{P}'(\Pet_k(\nu,\mu))$ is a basic tree if and only if each edge does not form a self-loop, which is equivalent to that $-\nu_{j+1}+j+k > -\mu_{j+1}+j$, i.e. $\lambda_{j+1}-\mu_{j+1}<k$.
        \item In $\mathcal{P}(\Pet_k(\lambda,\mu))$, the edges are initially pointed downward. Hence, $|\sigma|$, the number of edges that changed their directions, is the number of upward edges.

        The inversions of the permutation $\tau_\sigma$ never form in any pair of upward edges or downward edges. Thus, the inversion is the number of crossings happens when we move down the rows of upward edges in $\Pet_k(\lambda,\mu)$ one by one, starting from the largest index one. Let $\lambda_0=\nu\subset \lambda_1\subset\lambda_2\subset\cdots\subset \lambda_n=\lambda$ be the $k$-decomposition. When rotating the $i$-th upward edge, we move it down by one cell and move the corresponding row $\htt(\lambda_{n-i}/\lambda_{n-i-1})$ rows below. The result then follows.
    \end{enumerate}
    \end{proof}        
    As clusters are determined by $\mu$, the basic graphs with cluster induced by $\mu$ corresponds to the horizontal strips in $\nu/\mu$ fixing $\mu$. 

    The last piece of the dictionary is between good pairs of $\lambda/\mu$ and good orientations of $\Pet_k(\lambda,\mu)$. See \cref{fig:proper_correspondence} for examples.
    \begin{proposition}\label{prop:translation}
        The map from good orientations to their induced base shape $\nu$ and ribbons in $k$-decomposition is a bijection between good orientations $\sigma$ of $\mathcal{P}'(\Pet_k(\lambda,\mu))$ and $k$-good pairs $(\nu,\Theta)$ of $\lambda/\mu$.
    \end{proposition}
    \begin{proof}
        We check that
        \begin{itemize}
            \item This map is well-defined: The resulting $(\nu,\Theta)$ from the $k$-decomposition of $\sigma$ is a good pair.
            \item This map is injective;
            \item This map is surjective: For each good pair $(\nu,\Theta)$, the horizontal strip $\nu/\mu$ corresponds to a base graph, and adding back the ribbons in $\Theta$ induces a good orientation of $\mathcal{P}'(\Pet_k(\lambda,\mu))$.
        \end{itemize}
        $\mathcal{P}_\sigma=\mathcal{P}'(\Pet_k(\nu,\mu))$ is a basic graph, it follows that $\nu/\mu$ is a horizontal strip from \cref{prop:initial}. In the $k$-decomposition of $\sigma$, after we remove a ribbon $\lambda_i/\lambda_{i-1}$, we only remove blocks that are at the upper right of its starting block after. Hence, the left to its starting block is never removed and ribbons are left-justified.

        The injectivity of the correspondence is clear: different good orientations removes ribbons starting at different contents, yielding different $k$-decompositions. 

        Now, for each good pair $(\nu,\Theta)$, we can do the reversed operation of getting the $k$-decomposition, namely adding ribbons in $\Theta$ back one at a time, starting with ribbons with larger content. 
        
        Since we are adding ribbons with larger content first, no edges in $\mathcal{P}'(\Pet_k(\nu,\mu))$ need to move twice to get $\mathcal{P}'(\Pet_k(\lambda,\mu))$. Mark each moved edge as upward and the rest as downward, and we get an orientation of $\mathcal{P}'(\Pet_k(\lambda,\mu))$. The orientation is good since we start with a basic graph by \cref{prop:initial}, and each step only changes the position of the outward edge of a cluster. Surely, this good orientation induces $(\nu,\Theta)$.
    \end{proof}

    Having this correspondence in hand, we can conclude our results for $k$-Petrie numbers in terms of good pairs.

    \begin{proof}[Proof (of \cref{thm:refinedJJL})]
        \cref{thm:PetWeightInvo} tells us that there are two possibilities about the good orientations of $\mathcal{P}(\Pet_k(\lambda,\mu))$:
        \begin{enumerate}
            \item There is a unique good orientation for $\mathcal{P}(\Pet_k(\lambda,\mu))$. Let $(\nu,\Theta)$ be the corresponding good pair, then $\mathcal{P}(\Pet_k(\nu,\mu))$ is a tree that is also a basic graph. By \cref{prop:initial}, we know that $\nu/\mu$ is a horizontal strip whose rows have less than $k$ blocks.

            \item There are either $0$ or a power of $2$ good orientations for $\mathcal{P}(\Pet_k(\lambda,\mu))$. In this case, the weight of good orientations cancel each other out via involutions induced by cycles of $\mathcal{P}(\Pet_k(\lambda,\mu))$.
        \end{enumerate}
        
        On the other hand, from \cref{prop:translation}, we find that each good orientation $\sigma$ of $\mathcal{P}(\Pet_k(\lambda,\mu))$ corresponds to a good pair $(\nu,\Theta)$ with weight $w(\sigma)=\prod_{i=1}^{|\Theta|}(-1)^{\htt(\Theta_i)+1}$. The conclusion of \cref{thm:PetWeightInvo} implies that
        \[\det(\Pet_k(\lambda,\mu)) = \sum_{\sigma:\text{good orientation of }\Pet_k(\lambda,\mu)}w(\sigma) = \sum_{(\nu,\Theta):\text{good pair of }\lambda/\mu}\prod_{i=1}^{|\Theta|}(-1)^{1+\htt(\Theta_i)}.\]
    \end{proof}

    \section{Generating function of good pairs}\label{sec:GenFunc}
        In this section, we are going to prove \cref{thm:binomial}. To better track inversions, we require a stronger property of our $k$-Petrie matrices which we call \textit{totally Petrie}.
        
        \begin{definition}[totally Petrie]
            A $\{0,1\}$-matrix $A$ is \textit{totally Petrie} if for each $0$ in it, either all its upper-right entries or all its lower-left entries are $0$. 
        \end{definition}
        It is true that every $\Pet_k(\lambda,\mu)$ is totally Petrie. In fact, every totally Petrie matrix can be realized as $\Pet_k(\lambda,\mu)$ for some $\lambda,\mu,k$ via inserting suitable columns. See \cref{prop:extendTotallyPetrie}.
        
        If a Petrie matrix is totally Petrie, the effect of cycle involutions on $|\sigma|$ and $\inv(\tau_\sigma)$ is much more specific.
        \begin{theorem}\label{thm:InvoChangeWeight}
            Let $P$ be a totally Petrie matrix and $\sigma$ be a good orientation. Let $\varphi$ be an involution induced by the cycle $C$ of length $n$ of the Petrie graph.
            
            Then, 
            $|\varphi(\sigma)|-|\sigma|=\pm [2\nmid n]$ and 
            $\inv(\tau_{\varphi(\sigma)})-\inv(\tau_\sigma)=\pm [2\mid n]$,
            where the Iverson symbol $[P]$ is $1$ if $P$ is true, and $0$ otherwise.

            Furthermore, suppose that the first row of $\sigma$ that is changed by $\varphi$ is $r_i$. Then the signs in both formulas are positive if $r_i$ consists of $0$ and $1$'s, and is negative if $r_i$ consists of $0$ and $-1$'s.
        \end{theorem}
        The proof is somewhat technical and involves interesting enumerative structures. Therefore, we apply this theorem to obtain \cref{thm:binomial} before expanding with the proof in the next subsection.
        \begin{example}
            In \cref{fig:proper_correspondence}, from the first good orientation $\sigma_1$ to the second one $\sigma_2$, the cycle of clusters ([-1\,-2\,-3]\,[3\,4]\,[7\,8]\,[2]) is reversed. The cycle length is $4$. Since $2\mid4$, \cref{thm:InvoChangeWeight} predicts that \begin{enumerate*}[(1)]
                \item they have the same order $|\sigma_1|=|\sigma_2|$. Indeed, they both have $5$ upward edges.
                \item The first edge in the Petrie graph that is changed by the involution is the one pointed from -3 to 2, which is a upward (negative) edge in $\sigma_1$. Thus the theorem predicts that $\inv(\tau_{\sigma_2})$ is one less than $\inv(\tau_{\sigma_1})$.
                Indeed, by counting thew height of ribbons, $\inv(\tau_{\sigma_1})=1+1+1+1+2=6$, whereas $\inv(\tau_{\sigma_2})=1+1+1+1+1=5$.
            \end{enumerate*}
        \end{example}
        Note that the furthermore part says that the effect of each involution $\varphi$ on the order and $\inv$ applies independently, and commute with each other. 
        
        Given $k,\lambda,\mu$, we embed $\mathcal{P}'(\Pet_k(\lambda,\mu))$ into a cylinder of perimeter $k$ by curling the doubly-infinite strip of integer vertices onto it. Formally, let the cylinder be $S^1\times \R$, and we map integers by $t\mapsto (e^{\frac{-2\pi i t}{k}}, -\frac{t}{k})$. The vertex $\leq -\mu_1$ is put near $(0,\infty)$ and $\geq \lambda_1$ is put near $(0,-\infty)$. 
        
        Due to the shape of the cylinder, there are two kinds of cycles on $\mathcal{P}'(\Pet_k(\lambda,\mu))$:

        \begin{itemize}
            \item \textit{Contractible cycles}: Cycles that enclose a bounded area of the cylinder. Such a cycle is formed by several horizontal edges and an equal number of downward and upward edges. Hence, such a cycle has an even length.
 
            \item \textit{Non-contractible cycles}: Cycles that separate the top region and the bottom region. As a cycle does not cross itself, such a cycle will traverse one clockwise or counterclockwise round on the cylinder, and thus its number of downward edges and that of upward edges differ by $1$. Hence, such a cycle has an odd length.
        \end{itemize}
        
        Now we may define our statistics.
        \begin{definition}
            
        Let $\mathsf{c}(\lambda,\mu)$ be the number of contractible cycles in $\mathcal{P}'(\Pet_k(\lambda,\mu))$ and $\mathsf{nc}(\lambda,\mu)$ to be the number of non-contractible cycles. Also, let $\sigma_0$ be the good orientation with minimum number of upward edges and minimum number of $\inv(\tau_{\sigma})$ among all good orientations of $\mathcal{P}'(\Pet_k(\lambda,\mu))$.
        \end{definition}

        Now we can prove the expansion of $D_k(\lambda/\mu; t,q)$ in \cref{thm:binomial}. Recall that 
        \[D_k(\lambda/\mu; t,q)= \sum_{\sigma:\text{good}} t^{|\sigma|}q^{\inv(\sigma)} = \sum_{(\nu,\Theta):\text{ good pair of }\lambda/\mu} t^{|\Theta|}q^{\sum_{i=1}^q \htt(\Theta_i)},\]
        where in the first sum, $\sigma$ runs through all good orientations of $\mathcal{P}(\Pet_k(\lambda,\mu))$.
        
        \begin{proof}
            By \cref{thm:InvoChangeWeight}, every involution commute with each other, and if $\sigma_0$ is obtained from $\sigma$ by reversing $r$ contractible and $s$ non-contractible cycles, we have $t^{|\sigma|}q^{\inv(\tau_{\sigma})}=t^{|\sigma_0|+r}q^{\inv(\tau_{\sigma_0})+s}$. Hence, we obtain 
            \[D_k(\lambda/\mu; t,q)=\sum_{\sigma:\text{good orientation}}t^{|\sigma|}q^{\inv(\tau_{\sigma})}=t^{|\sigma_0|}q^{\inv(\tau_{\sigma_0})}(1+q)^{\mathsf{c}(\lambda,\mu)}(1+t)^{\mathsf{nc}(\lambda,\mu)}.\]

            If we want to classify good pairs $(\nu,\Theta)$ by $|\Theta|$, then by definition, $\sum_{n}D_{k,n}(\lambda/\mu) t^n =D_k(\lambda/\mu; t,1)$, which is in turn $t^{|\sigma_0|}\cdot 2^{\mathsf{c}(\lambda,\mu)}(1+t)^{\mathsf{nc}(\lambda,\mu)}$. Hence, $D_{k,n}=2^{\mathsf{c}(\lambda,\mu)} 
            \binom{\mathsf{nc}(\lambda,\mu)}{n-|\sigma_0|}$ when $|\sigma_0|\leq n\leq |\sigma_0|+\mathsf{nc}_k(\lambda,\mu)$.
        \end{proof}
    
        Note that even when $\lambda/\mu$ is connected, we cannot bound the number of non-contractible cycles.
        \begin{example}\label{ex:counterexample}
            Let $k=3$ and $\lambda/\mu=977764443/65553222$, see \cref{fig:counterexample}. Then $\mathcal{P}(\Pet_3(\lambda,\mu))$ has $3$ non-contractible cycles, which are $([2\,3],[5],[6\,7\,8])$, $([-4\,-5],[-2],[-1\,0\,1])$, and a self loop to the top vertex connecting cells $-9$ and $-6$. Indeed, $D_3(\lambda/\mu;t,q)=t^4q^4(1+t)^3$. 

            There are $8$ good pairs of $\mathcal{P}(\Pet_3(\lambda,\mu))$, $3$ of them has $5$ ribbons, also shown in \cref{fig:counterexample}. 
        \end{example}
        \begin{figure}
            \centering
            \[\begin{array}{ccc}
            \begin{tikzpicture}[baseline=-1cm, scale=0.9,y=0.7cm]
                \draw [red] (1.5,3.7)--(2.5,3.7)--(2.5,2.7)--(4.5,2.7);
				\node at (3.5,3) [above]{$\leq -10$};
				\draw (1.5,-3.7) -- (1.9,-3.7)--(1.9,-2.7)--(2.1,-2.7)--(2.1,-3.7)--(4.5, -3.7);
				\node at (3,-3.7) [below]{$\geq 9$};
                \draw (2,3) circle (6pt) node[inner sep=2pt] (-9){\scriptsize -9};
                \draw (4,2) circle (6pt) node[inner sep=2pt] (-8){\scriptsize -8};
                \draw (3,2) circle (6pt) node[inner sep=2pt] (-7){\scriptsize -7};
                \draw (2,2) circle (6pt) node[inner sep=2pt] (-6){\scriptsize -6};
				\draw (4, 1) circle (6pt) node[inner sep=2pt] (-5){\scriptsize -5};
				\draw (3, 1) circle (6pt) node[inner sep=2pt] (-4){\scriptsize -4};
				\draw (2, 1) circle (6pt) node[inner sep=2pt] (-3){\scriptsize -3};
				\draw (4, 0) circle (6pt) node[inner sep=2pt] (-2){\scriptsize -2};
				\draw (3, 0) circle (6pt) node[inner sep=2pt] (-1){\scriptsize -1};
				\draw (2, 0) circle (6pt) node[inner sep=2pt] (0){$0$};
				\draw (4,-1) circle (6pt) node[inner sep=2pt] (1){$1$};
				\draw (3,-1) circle (6pt) node[inner sep=2pt] (2){$2$};
				\draw (2,-1) circle (6pt) node[inner sep=2pt] (3){$3$};
				\draw (4,-2) circle (6pt) node[inner sep=2pt] (4){$4$};
				\draw (3,-2) circle (6pt) node[inner sep=2pt] (5){$5$};
				\draw (2,-2) circle (6pt) node[inner sep=2pt] (6){$6$};
				\draw (4,-3) circle (6pt) node[inner sep=2pt] (7){$7$};
				\draw (3,-3) circle (6pt) node[inner sep=2pt] (8){$8$};

                \draw[->] (-9) -- (-6);
                \draw[->] (-6) -- (-3);
                \draw[->] (-5) -- (-2);
                \draw[->] (-4) -- (-1);
                \draw[->] (-2) -- (1);
                \draw[->] (1) -- (4);
                \draw[->] (2) -- (5);
                \draw[->] (3) -- (6);
                \draw[->] (5) -- (8);

                \draw[red,-] (-8) -- (-7);
                \draw[red,-] (-7) -- (-6);           
                \draw[red,-] (-9) -- (1.5,3) --(1.2,2.7); \draw[red,-] (4.8,2.3)--(4.5,2)--(-8);
                \draw[red,-] (2.5,3) -- (-9);

                \draw[red,-] (-5) -- (-4);
                \draw[red,-] (-1) -- (0);
                \draw[red,-] (0) -- (1.5,0) --(1.2,-.3); \draw[red,-] (4.8,-.7)--(4.5,-1)--(1);
                \draw[red,-] (2) -- (3);
                \draw[red,-] (6) -- (1.5,-2) --(1.2,-2.3); \draw[red,-] (4.8,-2.7)--(4.5,-3)--(7);
                \draw[red,-] (7) -- (8);

            \end{tikzpicture}
            \qquad 
			\ytableaushort
			{
				\none\none\none\none\none\none {}{}{},
                \none\none\none\none\none {}{},
                \none\none\none\none\none {}{},
                \none\none\none\none\none {}{},
                \none\none\none {}{}{},
                \none\none{}{},
                \none\none{}{},
                \none\none{}{},
                {}{}{}}
            \end{array}\]
            \[\ytableaushort{
				\none\none\none\none\none\none {*(black)}{*(black)}{*(black)},
                \none\none\none\none\none {*(black)}{1},
                \none\none\none\none\none {1}{1},
                \none\none\none\none\none {2}{2},
                \none\none\none {*(black)}{*(black)}{2},
                \none\none{3}{3},
                \none\none{3}{4},
                \none\none{4}{4},
                {5}{5}{5}}
                \qquad
                \ytableaushort{
				\none\none\none\none\none\none {*(black)}{*(black)}{*(black)},
                \none\none\none\none\none {1}{1},
                \none\none\none\none\none {1}{2},
                \none\none\none\none\none {2}{2},
                \none\none\none {3}{3}{3},
                \none\none{*(black)}{4},
                \none\none{4}{4},
                \none\none{5}{5},
                {*(black)}{*(black)}{5}}
                \qquad 
                \ytableaushort{
				\none\none\none\none\none\none {1}{1}{1},
                \none\none\none\none\none {*(black)}{2},
                \none\none\none\none\none {2}{2},
                \none\none\none\none\none {3}{3},
                \none\none\none {*(black)}{*(black)}{3},
                \none\none{*(black)}{4},
                \none\none{4}{4},
                \none\none{5}{5},
                {*(black)}{*(black)}{5}}\]
            \caption{$\mathcal{P}(\Pet_3(\lambda,\mu))$ for $\lambda/\mu=977764443/65553222$, and all its good pairs $(\nu,\Theta)$ with $|\Theta|=5$.}
            \label{fig:counterexample}
        \end{figure}

    \subsection{Proof of \texorpdfstring{\cref{thm:InvoChangeWeight}}{Theorem 6.2}}\label{sec:proofMotzkin}
        We first investigate how $|\sigma|$ and $\inv(\tau_\sigma)$ are affected by the involution. 
        Since $P$ is totally Petrie, an inversion of $\tau_{\sigma}$ is only caused by a reversed row moving down and crossing another non-reversed row. Hence, the number of inversion related to one edge only depends on the position change of that edge. 
        
        For edges not in the cycle, their directions and positions are the same in both orientations $\sigma$ and $\varphi(\sigma)$, thus $\inv(\tau_{\varphi(\sigma)})-\inv(\tau_\sigma)$ is the difference of inversions that happens in the cycle.
        
        As the difference of $|\varphi(\sigma)|$ and $|\sigma|$ also only happen in the cycle, we can ignore everything not on the cycle, and it suffices to show the following:
        \begin{proposition}\label{lemma:cycleCase}
            Let $A$ be an $n\times(n-1)$ totally Petrie matrix whose Petrie graph is a cycle and induces two cyclic orientations $\sigma_1$ and $\sigma_2=\varphi(\sigma_1)$.\footnote{Namely, $\sigma_1,\sigma_2$ are good orientations of $A$ after appending a zero column vector to $A$, by which operation we give an isolated root vertex of $\mathcal{P}(A)$.}
            
            Suppose that $\sigma_2$ reverses the first row. Then
            
            \begin{itemize}
                \item $|\sigma_2|-|\sigma_1| = [2\nmid n]$.
                \item $\inv(\tau_{\sigma_2})-\inv(\tau_{\sigma_1})=[2\mid n]$.
            \end{itemize}
        \end{proposition}

        Let $\mathcal{CP}(n)$ denote the set of $n\times (n-1)$ Petrie matrices whose Petrie graph is a cycle. We will call such matrices \emph{cycle Petrie matrices}.
        
        We begin by illustrating some properties of cycle Petrie matrices. 
         
        \begin{lemma}\label{lemma:cycleCaseLemma}
            Let $A\in \mathcal{CP}(n)$ and suppose that $A$ has row vectors $v_i$. Then 
            \begin{enumerate}
                \item $\sum_i (-1)^i v_i=0$.
                \item The column vectors $v[a'_i,b'_i]$ of $A$ satisfies that $a'_{i+1}+b'_{i+1}=a'_{i}+b'_{i}+2$.
            \end{enumerate}
        \end{lemma}
        \begin{proof}\mbox{}\\
        \begin{enumerate}
            \item 
            Since the Petrie graph of $A$ is a cycle, we know that
            \[\sum_i \pm v_i=0\]
            for some appropriate choice of the signs. This implies that each column has an even number of $1$'s, which then implies $\sum_i (-1)^i v_i=0$ since $1$'s in each column appear consecutively.

            \item 
            For the second claim, since $v[a'_i,b'_i]\neq v[a'_{i+1},b'_{i+1}]$ and both $b'_i-a'_i, b'_{i+1}-a'_{i+1}$ are even numbers, along with the fact that $a'_n+b'_n\leq a'_1+b'_1 +2(n-1)$, we find that
            \[a'_n+b'_n\geq a'_{n-1}+b'_{n-1}+2\geq \cdots\geq a'_1+b'_1+2(n-1)\geq a'_n+b'_n.\]
            Therefore, these equalities hold and  $a'_i+b'_i$ must increase by exactly $2$ per step.
        \end{enumerate}
        \end{proof}

        \setcounter{MaxMatrixCols}{15}
        \begin{figure}[htbp!]
            \def \rr {\textbf{1}}
            \def \bb {\textbf{1}}
            \centering
            \[
                \scalebox{0.75}{$\begin{bmatrix}
                   1&\rr& & & & & & & & & \\
                   1&1&\bb& & & & & & & & \\
                    &1&1&\rr&\rr& & & & & & \\
                    &1&1&1&\bb&\bb& & & & & \\
                    & &1&1&1&\rr& & & & & \\
                    & & &1&1&1&\bb&\bb& & & \\
                    & & & &1&1&1&\rr& & & \\
                    & & & &1&1&1&1&\bb& & \\
                    & & & & &1&1&1&1&\rr& \\
                    & & & & & & &1&1&1& \\
                    & & & & & & &1&1&1&1\\
                    & & & & & & & & &1&1
				\end{bmatrix}$}
                \quad
                \leftrightarrow  
                \quad 
                \begin{tikzpicture}[scale=0.5]
                    \coordinate (prev) at (0,0);
                    \coordinate (U) at (1,1);
                    \coordinate (H) at (1,0);
                    \coordinate (D) at (1,-1);
                    \def \steps {U,H,H,U,H,D,U,D,H,D}
                    
                    \foreach \step in \steps
                    {\coordinate (A) at ($(prev) + (\step)$);
                    \draw[shorten >=2pt, shorten <=2pt, line width=1.5pt, line cap=round] (prev) -- (A);
                    \coordinate (prev) at (A);}

                    \def \red {(1,0), (3,0), (4,1), (5,0), (7,0), (9,0)}
                    \def \blue{(2,0), (4,0), (5,1), (6,0), (7,1), (8,0)}

                \end{tikzpicture}
            \]
            \caption{A matrix $A\in \mathcal{CP}(12)$ and its associated Motzkin path $\pi_A$}
            \label{fig:CyclePetrie}
        \end{figure}
        
        The second condition of \cref{lemma:cycleCaseLemma} actually tells us there is an injection from $\mathcal{CP}(n)$ to the set of \textit{Motzkin paths} of length $n-1$. Recall that a Motzkin path is a path from $(0,0)$ to $(0,n-1)$ consisting of three kinds of steps $U=(1,1)$ (up step), $D=(1,-1)$ (down step), and $H=(1,0)$ (horizontal step) whose $y$-coordinate never goes below $0$. 
        
        For each $A\in \mathcal{CP}(n)$, we associate a Motzkin path $\pi_A$ as the following: We let the $i$-th step of $\pi_A$ be
        \[\begin{cases}
            U, & \text{if } b'_i=b'_{i-1}+2,\, a'_i=a'_{i-1}; \\
            H, & \text{if } b'_i=b'_{i-1}+1,\, a'_i=a'_{i-1}+1; \\
            D, & \text{if } b'_i=b'_{i-1}\phantom{{}+1},\, a'_i=a'_{i-1}+2. 
        \end{cases}\]
        Equivalently, the height of the $i$-th step in the Motzkin path is $b'_{i+1}-i-1 = i+1-a'_{i+1}$. From this perspective, having no empty column tells us that the path never goes below $0$; $[a'_n,b'_n]=[n-2,n]$ implies that the endpoint is $(n-1,0)$.

        \begin{example}
            \cref{fig:CyclePetrie} shows an example of $\mathcal{CP}(12)$, whose corresponding permutation $\tau$ is $(0\,2\,6\,1\,5\,10\,7\,11\,9\,4\,8\,3)$ in cycle notation. The associated Motzkin path is $UHHUHDUDHD$, as displayed on the right.

            The heights of the Motzkin path are $1,1,1,2,2,1,2,1,1$, which correspond to the numbers of $1$'s in the upper triangular area of the matrix in each column.
        \end{example}
        
        To tell if a Motzkin path comes from a cycle Petrie matrices requires work. However, we have a basic property for any Motzkin path that is induced by a cycle Petrie matrix.

        \newcommand{\blue}{\textcolor{blue}{blue}}
        \newcommand{\red}{\textcolor{red}{red}}
        \begin{lemma}\label{lemma:MotzkinChar}
            Let $A\in \mathcal{CP}(n)$ and $\pi_A\in \{U,D,H\}^{n-1}$ be its associated Motzkin path. For each step $s$, we mark it \blue{} (even) if there is an even number of $H$'s before $s$ (including itself), and \red{} (odd) otherwise. Then for each $i$, 
            \begin{enumerate}
                \item The number of \red{} (odd) $D$ before step $i$ is at most that of \blue{} (even) $U$;
                \item Symmetrically, the number of \blue{} $D$ before step $i$ is at most that of \red{} $U$.
            \end{enumerate}
        \end{lemma}
        We denote this property as the \textit{\red-\blue} rule. \cref{fig:MotzkinColor} gives an example of the red-blue coloring.
        
        \begin{proof}
            We can classify the columns of $A$ by the parity of $a'_i$ (thus by $b'_i$). Note that the columns of $A$ span the particular subspace $\{(x_1,x_2,\dotsc,x_n)\in \R^n: \sum_{i=1}^n (-1)^i x_i=0\}$, and those columns with even $a'_i$'s should generate $\{(x_1,\dotsc,x_n)\in \R^n: x_{2k-1}=x_{2k} \forall k\}$, whereas those with odd $a'_i$'s should generate $\{(x_1,\dotsc,x_n)\in \R^n: x_1=0, x_{2k}=x_{2k+1} \forall k\}$.
        
            Now we derive these two conditions from the fact that the $\ell$-th even column should have $b'\geq 2\ell$, and the $\ell$-th odd column should have $b'\geq 2\ell+1$, otherwise they would not be linearly independent. 

            Let $u_i,d_i,h_i$ denote the number of $u$'s, $d$'s, $h$'s up to the $i$-th step, respectively, and $u_i^e,u_i^o,d_i^e,d_i^o,h_i^e,h_i^o$ denote their corresponding variants which only count steps in even(\blue)/odd(\red) up to the $i$-th step. 
            Using these symbols we have $b_{i+1}' =2+2u_i+h_i$. 
            By the definition of \blue/\red{} steps, we also have $h_i^e=\floor{\frac{h_i}{2}}, h_i^o=\ceil{\frac{h_i}{2}}$.

            If the $i$-th step is \red{} ($h_i$ is odd), then there are $\ell=u_i^o+d_i^o+h_i^o$ odd columns up to column $i$. The condition $b'\geq 2\ell+1$ becomes
            \[2+2u_i+h_i \geq 2(u_i^o+d_i^o+h_i^o)+1,\]
            which is $2u_i^e \geq 2d_i^o$, the first condition;
            
            If the $i$-th step is \blue{} ($h_i$ is even), then there are $\ell=1+u_i^e+d_i^e+h_i^e$ even columns up to column $i$. The condition $b'\geq 2\ell$ becomes 
            \[2+2u_i+h_i \geq 2(1+u_i^e+d_i^e+h_i^e),\]
            which is $2u_i^o\geq 2d_i^e$, the second condition.
        \end{proof}
        
            The \red-\blue{} rule is not a full characterization of cycle Petrie matrices. There are some Motzkin paths satisfying the \red-\blue{} rule, but are not associated with any of the elements in $\mathcal{CP}(n)$ (the smallest one being $UHUHDHD$). We will give one characterization in \cref{sec:detour}.

        There is an important fact about Motzkin paths satisfying such conditions that we will use:
        \begin{proposition}\label{prop:MotzkinParity}
            Let $\pi$ be a Motzkin path satisfying the \red-\blue{} rule. Consider the lattice points strictly below $\pi$ with $y\geq 0$, then there are an equal number of points with odd coordinate sums and with even coordinate sums.
        \end{proposition}
        \begin{figure}[htbp!]
            \def \rr {\textbf{\textcolor{red}{1}}}
            \def \bb {\textbf{\textcolor{green!70!black}{1}}}
            \centering
            \[
                \scalebox{0.75}{$\begin{bmatrix}
                   1&\rr& & & & & & & & & \\
                   1&1&\bb& & & & & & & & \\
                    &1&1&\rr&\rr& & & & & & \\
                    &1&1&1&\bb&\bb& & & & & \\
                    & &1&1&1&\rr& & & & & \\
                    & & &1&1&1&\bb&\bb& & & \\
                    & & & &1&1&1&\rr& & & \\
                    & & & &1&1&1&1&\bb& & \\
                    & & & & &1&1&1&1&\rr& \\
                    & & & & & & &1&1&1& \\
                    & & & & & & &1&1&1&1\\
                    & & & & & & & & &1&1
				\end{bmatrix}$}
                \leftrightarrow  
                \begin{tikzpicture}[scale=0.5]
                    \coordinate (prev) at (0,0);
                    \coordinate (U) at (1,1);
                    \coordinate (H) at (1,0);
                    \coordinate (D) at (1,-1);
                    \def \steps {U/blue,H/red,H/blue,U/blue,H/red,D/red,U/red,D/red,H/blue,D/blue}
                    
                    \foreach \step/\CoL in \steps
                    {\coordinate (A) at ($(prev) + (\step)$);
                    \draw[shorten >=2pt, shorten <=2pt, line width=1.5pt, line cap=round, color=\CoL] (prev) -- (A);
                    \coordinate (prev) at (A);}

                    \def \red {(1,0), (3,0), (4,1), (5,0), (7,0), (9,0)}
                    \def \blue{(2,0), (4,0), (5,1), (6,0), (7,1), (8,0)}
                    \foreach \r in \red{ \fill[red] \r+(0,0.4) circle (3pt); }
                    \foreach \b in \blue{ \draw[green!70!black] \b+(0,0.4) circle (3pt); }

                    \draw[dashed] (3,1)--(6,1);
                    \draw[dashed] (6,1)--(7,0)--(8,1);
                \end{tikzpicture}
            \]
            \caption{A matrix $A\in \mathcal{CP}(12)$, its associated Motzkin path $\pi_A$ with red-blue colors, and possible removals of peaks. Points are drawn slightly higher than they should be for aesthetic purposes. }
            \label{fig:MotzkinColor}
        \end{figure}
        \begin{proof}
            We prove by induction on the number of lattice points below the path. Let $U_0$ be the highest up step and $D_0$ be the first down step after $U_0$. Between $U_0$ and $D_0$, there are $h$ $H$'s and no other up/down steps.
            \begin{itemize}
                \item If $h$ is even, then $U_0$ and $D_0$ are steps of the same color, we dent this peak by switching $D_0$ and $U_0$. This removes $h+1$ even lattice points and $h+1$ odd lattice points.

                \item If $h$ is odd, then $U_0$ and $D_0$ are of different color. We flatten this peak by changing both $U_0$ and $D_0$ into flat steps. This removes $h+1$ lattice points, half of which are even.
            \end{itemize}
            In either case, we removed the same number of even and odd lattice points, and the remaining Motzkin path still satisfies the \red-\blue{} rule. 

            We eventually transform the Motzkin path into the all-flat Motzkin path, in which case the condition trivially holds.
        \end{proof}
        \begin{example}
            \cref{fig:MotzkinColor} follows the example of \cref{fig:CyclePetrie}. The red/green points represent points with odd/even coordinate sums.

            To the upper triangular area of the Petrie matrix, we color the $1$'s into red/green by reading the color of red/green points by column. The $1$'s in the same row get the same color.
            
            To show that there are the same number of red/green points, one first flattens the first peak from the $4$th to $6$th step, flips the peak at the $7$th and $8$th steps, then finally flattens the last two peaks to get the all-flat Motzkin path.
        \end{example}
        
        With this lemma, we can finish our proof for \cref{lemma:cycleCase}, which completes the proof of \cref{thm:InvoChangeWeight}.
        \begin{proof}
            By \cref{lemma:cycleCaseLemma}, the edges are reversed alternately. Hence, the orientation starting with a reverse edge gets $0$ or $1$ more reversed edge(s), giving the first assertion.

            For the second assertion, when $A$ is formed by rows $v[a_1,b_1],v[a_2,b_2],\dotsc,v[a_n,b_n]$, we first claim that 
            \[\inv(\tau_{\sigma_2})-\inv(\tau_{\sigma_1}) = \sum_{k=1}^{n-1} (-1)^{k-1}(b_k-k)+[2\mid n].\] 
            Note that the term $\sum_{k=1}^{n-1} (-1)^{k-1}(b_k-k)$ corresponds to the number difference of the red and green colored $1$'s. By \cref{prop:MotzkinParity}, this term is $0$ and hence the second assertion follows.
            
            \begin{proof}[Proof (of claim)]
                Due to the totally Petrie shape of $A$, the permutation $\tau_{\sigma_i}$ is increasing on the set of excedence indices\footnote{An index $i$ of $\tau$ is an \emph{excedence} if $\tau_i>i$.} and also increasing in indices of nonexcedence. For permutations (of $0,1,\dotsc,n-1$) with this property, the inversion number is 
            \[\inv(\tau)=\sum_{\substack{0\leq i\leq n-1\\ \tau_i > i}} \tau_i-i, \]
            since there are $\tau_i$ entries being smaller than $\tau_i$, with $i$ of which before $\tau_i$.

            Applying to $\tau_{\sigma_1}$, which has excedence indices $1,3,5,\dotsc$, and $\tau_{\sigma_2}$, which has excedence indices $0,2,4,\dotsc$, we find that 
            \[\inv(\tau_{\sigma_2})-\inv(\tau_{\sigma_1}) = \sum_{0\leq i\leq n-1} (-1)^i (b_{i+1}-i) =\sum_{1\leq k\leq n}(-1)^{k-1}(b_k-k+1).\]
            Note that the rows start with index $1$. As $b_n=n-1$ and $\sum_{1\leq k\leq n-1}(-1)^{k-1} = [2\mid n]$, the claim follows.
                \let\qedsymbol\relax
            \end{proof}

        \end{proof}    

	\section{Proof of Plethystic Pieri rule}\label{sec:Plethysm}
        
        The last of our results is to give another proof of the plethystic Pieri rule (\cref{thm:plethysticPieri}) for $E(t)$ using \cref{thm:refinedJJL}.

        Recall from \cref{sec:Symbols} that 
        \[\sum_{n=0}^\infty G(k,n) t^n = \prod_{i} (1+x_it+\cdots+x_i^{k-1}t^{k-1}) = \prod_i \frac{1-x_i^kt^k}{1-x_it}= H(t)\cdot E_k(-t^k),\]
        where $E_k(t)$ is the generating function of plethystic elementary symmetric function $e_n\circ p_k$, i.e. $E_k(t)=\sum_{n=0}^\infty (e_n\circ p_k)t^n=\prod_{i}(1+x_i^kt)$.

        \begin{proof}[Proof (of plethystic Pieri rule)]
            We will show that $E_k(u)\cdot s_\mu = \sum_{\Theta} (-1)^{\sum_{i=1}^n\htt(\Theta_i)}s_{\mu+\Theta} u^{|\Theta|}$, where $\Theta$ runs through all the ways of adding several left-justified $k$-ribbons of $\mu$, and $\mu+\Theta$ denotes the shape obtained by $\mu$ adding ribbons in $\Theta$.

            By \cref{thm:refinedJJL}, 
            \begin{align*}
                E_k(-t^k)\cdot H(t)\cdot s_\mu 
                &= \left(\sum_{n\geq 0}G(k,n)t^n\right) s_\mu \\&=\sum_{\nu}\sum_{\Theta} (-1)^{|\Theta|+\sum_{i=1}^{|\Theta|}\htt(\Theta_i)}s_{\mu+\Theta} t^{|\nu/\mu|+k|\Theta|},
            \end{align*}
            where the sum runs through pairs of $(\nu,\Theta)$ such that:
            \begin{itemize}
                \item $\nu/\mu$ is a horizontal strip;
                \item $\Theta$ is a sequence of left-justified $k$-ribbons that is based on $\nu$.
            \end{itemize}

            To reverse the operation of adding horizontal trip, we consider $E(-t) s_\mu$ rather than $s_\mu$. On one hand, we have $H(t)\cdot E(-t)\cdot s_\mu=s_\mu$; On the other hand, by (ordinary and dual) Pieri rules we expand it into $\sum_{\kappa}\sum_{\nu} s_\nu t^{|\nu/\kappa|}(-t)^{|\kappa/\mu|}$, where $\kappa,\nu$ runs through the shapes such that $\kappa/\mu$ is a vertical strip\footnote{Vertical strips are skew shapes where each row has at most one box.} and $\nu/\kappa$ is a horizontal strip.

            When we change $s_\mu$ to $E(-t)\cdot s_\mu$, we can get
            \begin{align*}
                E_k(-t^k) \cdot s_\mu&=(E_k(-t^k)\cdot H(t))\cdot E(-t)\cdot s_\mu\\
                &=(E_k(-t^k)\cdot H(t))\sum_{\kappa}s_\kappa(-t)^{|\kappa/\mu|}\\
                &=
                \sum_{\kappa}\sum_{\nu}\sum_{\Theta} (-1)^{|\Theta|+\sum_{i=1}^{|\Theta|}\htt(\Theta_i)}s_{\nu+\Theta}\, t^{k|\Theta|} t^{|\nu/\kappa|}(-t)^{|\kappa/\mu|},
                \end{align*}
            where 
            \begin{itemize}
                \item $\kappa,\nu$ runs through the shapes such that $\kappa/\mu$ is a vertical strip and $\nu/\kappa$ is a horizontal strip.
                \item $\Theta$ runs through the sequences of left-justified $k$-ribbons that is based on $\nu$.
            \end{itemize}
              
            Before adding $\Theta$, we know that the only nonvanishing term after adding $\nu$ and $\kappa$ is $s_\mu$ with coefficient $1$. Hence, the sum is also equal to
              \[\sum_{\Theta} (-1)^{|\Theta|+\sum_{i=1}^{|\Theta|}\htt(\Theta_i)}s_{\mu+\Theta} t^{k|\Theta|}=\sum_{\Theta} (-1)^{\sum_{i=1}^{|\Theta|}\htt(\Theta_i)}s_{\mu+\Theta}\cdot (-t^k)^{|\Theta|}
            \]
            Replacing the variable $-t^k$ with $u$ and the result follows.
        \end{proof}

        Applying $\omega$ involution to \cref{thm:plethysticPieri}, we get the Plethystic Pieri rule for $(h_n\circ p_k)$:
        \begin{corollary}
            Let $\lambda$ be a partition and $n,k\in \N$. Then
            \[(h_n\circ p_k)s_\mu = \sum_{\lambda} (-1)^{\sum_{i=1}^n\htt(\Theta_i)}s_\lambda,\]
            where $\lambda$ runs through the partitions obtained by $\mu$ adding $n$ top-justified $k$-ribbons $\Theta_1,\Theta_2,\dotsc,\Theta_n$.
        \end{corollary}
        \begin{proof}
            As $\omega(e_n\circ p_k)=h_n\circ p_k$, we get 
            \[(h_n\circ p_k) s_\mu= \sum_{\Theta} (-1)^{\sum_{i=1}^n\htt(\Theta_i)}s_{(\mu'+\Theta)'},\]
            where $\Theta$ runs through the set of $n$ left-justified $k$-ribbons adding on $\mu'$. Equivalently, the sum is
            \[\sum_{\Theta} (-1)^{\sum_{i=1}^n\htt(\Theta_i)}s_{\mu+\Theta},\]
            where $\Theta$ runs through the set of $n$ top-justified $k$-ribbons adding on $\mu$. 
        \end{proof}
    \section{Concluding remark}\label{sec:Remark}
	\cref{thm:PetWeightInvo} provided a weighted sum formula to calculate the determinant of Petrie matrices. The benefit of the formula is that the criterion of being a good orientation can be checked locally (each non-last vertex has outdegree one), whereas the treeness of a graph is a global criterion. Using this,  \cref{thm:refinedJJL} gave a combinatorial formula for $\pet_k(\lambda,\mu)$.

        We also explained the appearance of power-of-2 phenomenon in the number of $k$-good pairs of $\lambda/\mu$ with $n$ ribbons considered in \cite{jin2024pierilikerulepetriesymmetric} via \cref{thm:binomial}. 
        An interesting question would be determining the statistics $\mathsf{c}(\lambda,\mu)$ and $\mathsf{nc}(\lambda,\mu)$ directly from the partition $\lambda,\mu$.

        Lastly, we gave from Petrie symmetric functions an alternative of the plethystic Pieri rule.
	

        \subsection{More on cycle Petrie matrices}\label{sec:detour}
        In \cref{sec:proofMotzkin} we introduced a new combinatorial object called \emph{cycle Petrie matrices}. Elements in $\mathcal{CP}(n)$ correspond to many other existing structures in the literature, but with some modification. We now discuss all we have discovered so far.

        \begin{proposition}\label{prop:extendTotallyPetrie}
            Any totally Petrie square matrix can be written in the form $\Pet_k(\lambda,\mu)$.
        \end{proposition}
        \begin{proof}
            It suffices to show that every totally Petrie matrix is a submatrix of a totally Petrie matrix whose rows each have the same number of $1$'s.

            Let $\alpha_i$ be the number of $1$'s in row $i$. We induct on the variation $v=\sum_{i=1}^{n-1}|\alpha_i-\alpha_{i+1}|$. $v=0$ when every row has the same number of $1$'s.

            Otherwise, if $\alpha_i>\alpha_{i+1}$, we add a column to the matrix, after the end of $1$ in row $i$, with $1$'s starting at row $i+1$ and ending at some row $i'$ so that the matrix is still a totally Petrie matrix. See \cref{fig:extendTotallyPetrie} for an example. If $\alpha_{i'}<\alpha_{i'+1}$, then we now decreased the variation by $2$. If not, then we repeat this procedure by adding another column of $1$'s at the place where row $i'$ ends its $1$'s, with $1$'s in indices $i'+1,i'+2,\dotsc,i''$ so that it remains totally Petrie. We do this until we stop or run out of rows; in the latter case, the variation drops by $1$.
            
            The case $\alpha_i<\alpha_{i+1}$ is similar.
        \end{proof}
        \begin{figure}[htbp!]
            \centering
            \[
                \scalebox{0.75}{$\begin{bmatrix}
                   1&1& & & & & & & & & \\
                   1&1&1& & & & & & & & \\
                    &1&1&1&1& & & & & & \\
                    &\textcolor{red}1&\textcolor{red}1&\textcolor{red}1&\textcolor{red}1&\textcolor{red}1& & & & & \\
                    & &\textcolor{red}1&\textcolor{red}1&\textcolor{red}1&\textcolor{red}1& & & & & \\
                    & & &1&1&1&1&1& & & \\
                    & & & &1&1&1&1& & & \\
                    & & & &1&1&1&1&1& & \\
                    & & & & &1&1&1&1&1& \\
                    & & & & & & &1&1&1& \\
                    & & & & & & &1&1&1&1\\
                    & & & & & & & & &1&1
				\end{bmatrix}$}
                \quad
                \implies  
                \quad 
                \scalebox{0.75}{$\begin{bmatrix}
                   1&1& & & & & & & & & \\
                   1&1&1& & & & & & & & \\
                    &1&1&1&1& & & & & & \\
                    &\textcolor{red}1&\textcolor{red}1&\textcolor{red}1&\textcolor{red}1&\textcolor{red}1& & & & & \\
                    & &\textcolor{red}1&\textcolor{red}1&\textcolor{red}1&\textcolor{red}1&\textcolor{green!80!black}1 & & & & \\
                    & & &1&1&1&\textcolor{green!80!black}1&1&1& & & \\
                    & & & &1&1&\textcolor{green!80!black}1&1&1& & & \\
                    & & & &1&1&\textcolor{green!80!black}1&1&1&1& & \\
                    & & & & &1&\textcolor{green!80!black}1&1&1&1&1& \\
                    & & & & & & & &1&1&1&\textcolor{green!80!black}1 \\
                    & & & & & & & &1&1&1&\textcolor{green!80!black}1&1\\
                    & & & & & & & & & &1&\textcolor{green!80!black}1&1
				\end{bmatrix}$}
            \]
                \caption{One induction step in \cref{prop:extendTotallyPetrie}}
            \label{fig:extendTotallyPetrie}
        \end{figure}
        
        \subsubsection{Nonnesting cycle permutation}
            We say a permutation $\pi$ is \textit{nonnesting} if for each distinct $i,j$, let $\{a_1,a_2\}=\{i,\pi(i)\}$ and $\{b_1,b_2\}=\{j,\pi(j)\}$, where $a_1\leq a_2,b_1\leq b_2$, then $a_1<b_1\leq b_2<a_2$ does not happen. 

            For $n\geq 3$, each matrix in $\mathcal{CP}(n)$ corresponds to \textbf{$2$} nonnesting cycle permutations of order $n$, one for each orientation of the cycle.

            \cref{lemma:cycleCase} can then be translated into the following corollary:
            \newcommand{\spin}{\mathsf{spiral}^+}
            \begin{corollary}\label{Cor:cycleCaseCor}
                Given any permutation $\sigma$, we define a statistic $\spin(\sigma)$ to be the number of index pairs $(i,j)$ such that $i\leq \sigma(j)<\sigma(i)\leq j$.

                Then for a nonnesting cycle permutation $\pi$ of length $n$, we have $|\spin(\pi)-\spin(\pi^{-1})|=[2|n]$.
            \end{corollary}
            \begin{proof}
                Consider the orientation $\sigma$ of a cycle Petrie matrix corresponding to $\pi$. Then we have $\inv(\tau_\sigma)=\spin(\pi_\sigma)$.

                The result follows from the fact that $\pi_{\varphi(\sigma)}=\pi^{-1}$ and that $|\inv(\tau_\sigma)-\inv(\tau_{\varphi(\sigma)})|=[2|n]$ by \cref{lemma:cycleCase}.
            \end{proof}
            In terms of this setting, an index pair $(i,j)$ is an inversion of $\tau_\sigma$ if and only if 
            \begin{itemize}
                \item $i$ is excedence and $j$ is a nonexcedence: $\tau_i>\pi^{-1}(\tau_i)$ and $\tau_j<\pi^{-1}(\tau_j)$,
                \item $\tau_i>\tau_j$.
            \end{itemize}

        \subsubsection{Polygonal line}
            We can visualize the cycle $\tau_\sigma$ by drawing a polygonal line $L(\tau_\sigma)$ in the following way: Let $\tau_\sigma=(c_1c_2c_3\cdots c_n)$ with $c_1=0$, we connect the points $(a_1,c_1), (a_2,c_2),\dotsc,(a_n,c_n),(a_{n+1},c_1)$, where $a_1=0$, and for $i\geq 1$, 
            \[a_{i+1}=\begin{cases}
                a_i+1, &\text{if } c_{i+1}>c_i\\
                a_i-1, &\text{if } c_{i+1}<c_i.
            \end{cases}\] 
        That is, $a_i$ records the difference of prefix ascents and descents along the path. 

            Each matrix in $\mathcal{CP}(n)$ corresponds to such a polygonal line, such that, when projecting every segment onto the $x$-axis, no segment lies in the interior of another segment.
        \begin{example}
            \cref{fig:polygonal} shows the same $A\in \mathcal{CP}(12)$ as in \cref{fig:CyclePetrie}. The corresponding cycle is $(0\,2\,6\,1\,5\,10\,7\,11\,9\,4\,8\,3)$. The diagram on the right is $L(\tau_{\sigma})$ with $y$-axis stretched $3\times$ longer.
        \end{example}

        \begin{figure}
            \centering
            \[
                \scalebox{0.75}{$\begin{bmatrix}
                   1&1& & & & & & & & & \\
                   1&1&1& & & & & & & & \\
                    &1&1&1&1& & & & & & \\
                    &1&1&1&1&1& & & & & \\
                    & &1&1&1&1& & & & & \\
                    & & &1&1&1&1&1& & & \\
                    & & & &1&1&1&1& & & \\
                    & & & &1&1&1&1&1& & \\
                    & & & & &1&1&1&1&1& \\
                    & & & & & & &1&1&1& \\
                    & & & & & & &1&1&1&1\\
                    & & & & & & & & &1&1
				\end{bmatrix}$}
                \leftrightarrow  
                \begin{tikzpicture}[scale=0.5, baseline=2.25cm]
                    \draw[fill=blue!20]
                        (0,0)
                        --(2,3) 
                        --(6,6) 
                        --(1,3) 
                        --(5,6) 
                        --(10,9) 
                        --(7,6) 
                        --(11,9) 
                        --(9,6) 
                        --(4,3) 
                        --(8,6) 
                        --(3,3)  
                        --cycle;
                       \fill (0,0) circle (3pt);
                       \fill (2,3) circle (3pt);
                       \fill (6,6) circle (3pt);
                       \fill (1,3) circle (3pt);
                       \fill (5,6) circle (3pt);
                       \fill (10,9) circle (3pt);
                       \fill (7,6) circle (3pt);
                       \fill (11,9) circle (3pt);
                       \fill (9,6) circle (3pt);
                       \fill (4,3) circle (3pt);
                       \fill (8,6) circle (3pt);
                       \fill (3,3) circle (3pt);
                \end{tikzpicture}
            \]
            \caption{An example $A\in \mathcal{CP}(12)$ and its corresponding $L(\tau_{\sigma})$.}
            \label{fig:polygonal}
        \end{figure}
        While being a polygonal line, it satisfies some other properties:
        \begin{lemma}
            Let $A\in \mathcal{CP}(n)$ and $L(\tau_{\sigma})$ be its corresponding polygonal line. Then
            \begin{enumerate}
            \item The polygonal line $L(\tau)$ does not intersect itself.
            \item If $A$ has size $n\times (n-1)$ and $n$ is even, then $L(\tau)$ is a closed curve, i.e. $a_{n+1}=0$.
            \end{enumerate}
        \end{lemma}
        \begin{proof}
           Notice that the slope of each segment in the polygonal line is positive. For two of the segments to intersect, one's endpoints have coordinates strictly between the other's. This means that in $A$, there is a row whose $1$'s start later but end earlier than another row, which cannot occur in $A$.
        
           When $n$ is even, by \cref{lemma:cycleCaseLemma}, we have an equal number of reversed rows and non-reversed rows in $\sigma$, hence $a_{n+1}=0$.

        \end{proof}

        It is also likely true that if $n$ is even, the coordinates of the endpoints of $L$ are weakly increasing along with the $x$-coordinate. One possible proof is by writing $A$ as $\Pet_k(\lambda,\mu)$ for some $k,\lambda,\mu$ and discuss the topology of the enclosed regions of $\mathcal{P}'(\Pet_k(\lambda,\mu))$. It would be nice if a clean proof exists.
        \subsubsection{\red-\blue{} Motzkin paths, 2-colored noncrossing set partitions, and 1234-avoiding permutations} 
            We also showed that each matrix in $\mathcal{CP}(n)$ corresponds to a Motzkin path of length $n-1$, satisfying the \red-\blue{} rule described in \cref{lemma:MotzkinChar}. While this class is strictly larger than Motzkin paths induced by cycle Petrie matrices, we found more relations between other objects and \red-\blue{} rule Motzkin paths, rather than the cycle Petrie ones.
            
            Let $a_n$ be the number of Motzkin paths of length $n$ satisfying the \red-\blue{} rule. The first few terms are $1,1,2,3,6,11,23,47,103,225, 513,1173, 2761, 6529, 15767, 38265$, $94359, 233795, 586590, 1478295, \ldots$. Surprisingly, its odd terms follow A005802, and its even terms follow A216947. Both of them have combinatorial interpretations and are yet to have a Motzkin path interpretation. 
            The $n$-th term of A005802 counts the number of permutations in $S_n$ with the longest increasing subsequence of length $\leq 3$, 
            whereas A216947 counts the $2$-coloured noncrossing set partitions of $[n]$. 
            Both of the sequences are also related to invariant theory.
            For A216947, see, for example, \cite[A216947]{oeis} and \cite{bostan2019sequencesassociatedinvarianttheory,bostan2022combinatorialsequencesassociatedinvariant}; 
            For A005802, see \cite[A005802]{oeis}.

            This coincidence is true by numerical check \cite[A389602]{oeis}. As this is out of the scope of our paper, we leave the question here for interested readers.
            \begin{question}
                 For any $n\geq 1$, Find a combinatorial proof that the $n$-th term of A005802 (which starts with $0$-th term) corresponds to $a_{2n-1}$, and the $n$-th term of A216947 corresponds to $a_{2n}$.
            \end{question}

        \subsubsection{Combinatorial test for being a Motzkin path of \texorpdfstring{$\mathcal{CP}(n)$}{CP(n)}}
        For those Motzkin paths that correspond to an element in $\mathcal{CP}(n)$, although there are no known quick checking methods, there is a reflection method to check whether a Motzkin path satisfies the rule. To this purpose, we need to transform our Motzkin paths first:
        \begin{description}
            \item[Step 1:] We attach an up edge and a down edge before and after the Motzkin path, respectively. 
            \item[Step 2:] We push diagonals outward to get axis paths: We replace each up edge from $(i,j)$--$(i+1,j+1)$ into $(i,j)$--$(i,j+1)$--$(i+1,j+1)$, and each down edge from $(i,j)$--$(i+1,j-1)$ into $(i,j)$--$(i+1,j)$--$(i+1,j-1)$.
            \item[Step 3:] We connect $(0,0)$ to $(n+1,0)$ to enclose a region.
        \end{description}  

        \begin{proposition}
            A motzkin path of length $n-1$ corresponds to an element in $\mathcal{CP}(n)$ if and only if the ray starting from $(\frac{1}{2},0)$ with angle $45^\circ$, when reflected by the segments of the motzkin path, touches every unit segment before going back to $(0,\frac{1}{2})$. 
        \end{proposition}
        We sketch the proof via an example.
        \begin{example}
            We still consider the matrix in $\mathcal{CP}(12)$ with corresponding permutation $(0\,2\,6\,1\,5\,10\,7\,11\,9\,4\,8\,3)$. We can shift the columns to obtain a matrix with a horizontal symmetric axis (shown red in \cref{fig:ray-shooting}). We can then read the Motzkin path by both the upper border and the lower border of ones, and we concatenate one up step before and one down step after the Motzkin path.

            Note that in the upper-right diagram, all vertical lines correspond to a reflection with respect to the red line. In the permutation, $10$ maps to $7$, which can be seen in both the matrix and the reflection route in the Motzkin path. This is a direct correspondence.
        \end{example}
        \begin{figure}
            \centering
            \[
            \setlength{\arraycolsep}{4pt}
                \scalebox{0.75}{$\begin{bmatrix}
                   1&1& & & & & & & & & \\
                   1&1&1& & & & & & & & \\
                    &1&1&1&1& & & & & & \\
                    &1&1&1&1&1& & & & & \\
                    & &1&1&1&1& & & & & \\
                    & & &1&1&1&1&1& & & \\
                    & & & &1&1&1&1& & & \\
                    & & & &1&1&1&1&1& & \\
                    & & & & &1&1&1&1&1& \\
                    & & & & & & &1&1&1& \\
                    & & & & & & &1&1&1&1\\
                    & & & & & & & & &1&1
				\end{bmatrix}$}
                \hspace{-11.5em}  
                \begin{tikzpicture}[xscale=0.35, yscale=0.32, baseline=1.65cm]
                    \draw[blue!50]
                        (0,11)
                        --(2,11)--(2,7) 
                        --(6,7) .. controls (6.3,7.5).. (6,8) 
                        --(1,8) .. controls (.7,8.5).. (1,9) 
                        --(5,9) --(5,3) -- (10,3);
                    \draw[olive,line width=.8pt]
                        (10,3) .. controls (10.3,2.5).. (10,2) 
                        --(7,2) .. controls (6.7,1.5).. (7,1);
                    \draw[blue!50,->] (7,1)
                        --(11,1) .. controls (11.3,0.5).. (11,0)
                        --(9,0) --(9,4)
                        -- (6,4);
                    \draw[red] (-.5,11.5)--(11.5,-.5);
                \end{tikzpicture}
                \quad \longrightarrow \quad 
                \scalebox{0.75}{$\begin{bmatrix}
                    & & & & & & & & & & \\
                    & & & & & & & & & & \\
                    & & & & & & & & & & \\
                    & & & &1&1& &1& & & \\
                    &1&1&1&1&1&1&1&1&1& \\
                   1&1&1&1&1&1&1&1&1&1&1\\
                   1&1&1&1&1&1&1&1&1&1&1\\
                    &1&1&1&1&1&1&1&1&1& \\
                    & & & &1&1& &1& & & \\
                    & & & & & & & & & & \\
                    & & & & & & & & & & \\
                    & & & & & & & & & & 
				\end{bmatrix}$}
                \hspace{-11.5em}  
                \begin{tikzpicture}[xscale=0.35, yscale=0.32, baseline=1.8cm]
                    \draw[blue!50]
                        (0,6)
                        --(2,8)--(2,4) 
                        --(6,8) .. controls (6.5,8.5).. (6,9) 
                        --(1,4) .. controls (.5,4.5).. (1,5) 
                        --(5,9) --(5,3)--(10,8);
                    \draw[olive,line width=0.8pt]
                        (10,8) .. controls (10.5,7.5).. (10,7) 
                        --(7,4) .. controls (6.5,3.5).. (7,3);
                    \draw[blue!50,->] (7,3)
                        --(11,7) .. controls (11.5,6.5).. (11,6)
                        --(9,4) --(9,8)
                        -- (6,5);
                    \draw[red, line width=.8pt] (-.5,6)--(11.5,6);
                    
                    \draw[shorten >=.2pt, shorten <=.2pt, line width=.8pt, line cap=round] (-.5,6).. controls (-.5,7) ..(.5,7);
                    \draw[orange, shorten >=0.2pt, shorten <=.2pt, line width=.8pt, line cap=round] (.5,7).. controls (.5,8) ..(1.5,8)
                    --(2.5,8)
                    --(3.5,8)
                    .. controls (3.5,9) .. (4.5,9)
                    --(5.5,9)
                    .. controls (6.5,9) ..(6.5,8)
                    .. controls (6.5,9) ..(7.5,9)
                    .. controls (8.5,9) ..(8.5,8)
                    --(9.5,8)
                    .. controls (10.5,8)..(10.5,7);
                    \draw[shorten >=.2pt, shorten <=.2pt, line width=.8pt, line cap=round] (10.5,7).. controls (11.5,7)..(11.5,6);

                    \draw[shorten >=.2pt, shorten <=.2pt, line width=.8pt, line cap=round] (-.5,6).. controls (-.5,5) ..(.5,5);
                    \draw[orange, shorten >=0.2pt, shorten <=.2pt, line width=.8pt, line cap=round] (.5,5).. controls (.5,4) ..(1.5,4)
                    --(2.5,4)
                    --(3.5,4)
                    .. controls (3.5,3) .. (4.5,3)
                    --(5.5,3)
                    .. controls (6.5,3) ..(6.5,4)
                    .. controls (6.5,3) ..(7.5,3)
                    .. controls (8.5,3) ..(8.5,4)
                    --(9.5,4)
                    .. controls (10.5,4)..(10.5,5);
                    \draw[shorten >=.2pt, shorten <=.2pt, line width=.8pt, line cap=round] (10.5,5).. controls (11.5,5)..(11.5,6);

                \end{tikzpicture}
            \]
            \[
            \setlength{\arraycolsep}{4pt}
                \longrightarrow\hspace{2em}
                \scalebox{0.75}{$\begin{matrix}
                    & & & & & & & & & & \\
                    & & & & & & & & & & \\
                    & & & & & & & & & & \\
                    & & & &1&1& &1& & & \\
                    &1&1&1&1&1&1&1&1&1& \\
                   1&1&1&1&1&1&1&1&1&1&1\\
                    & & & & & & & & & & \\
				\end{matrix}$}
                \hspace{-11em}  
                \begin{tikzpicture}[xscale=0.35, yscale=0.32, baseline=2.6cm]
                    \draw[blue!50]
                        (0,6)
                        --(2,8)--(4,6) 
                        --(6,8) .. controls (6.5,8.5).. (6,9) 
                        --(3,6)--(1,8) .. controls (.5,7.5).. (1,7) 
                        --(2,6)--(5,9) --(8,6)--(10,8);
                    \draw[olive,line width=0.8pt]
                        (10,8) .. controls (10.5,7.5).. (10,7) 
                        --(9,6)--(7,8) .. controls (6.5,8.5).. (7,9);
                    \draw[blue!50,->] (7,9) --(10,6)
                        --(11,7) .. controls (11.5,6.5).. (11,6)
                        --(9,8) -- (7,6) 
                        -- (6,7);
                    \draw[red, line width=.8pt] (-.5,6)--(11.5,6);
                    
                    \draw[shorten >=.2pt, shorten <=.2pt, line width=.8pt, line cap=round] (-.5,6).. controls (-.5,7) ..(.5,7);
                    \draw[orange, shorten >=0.2pt, shorten <=.2pt, line width=.8pt, line cap=round] (.5,7).. controls (.5,8) ..(1.5,8)
                    --(2.5,8)
                    --(3.5,8)
                    .. controls (3.5,9) .. (4.5,9)
                    --(5.5,9)
                    .. controls (6.5,9) ..(6.5,8)
                    .. controls (6.5,9) ..(7.5,9)
                    .. controls (8.5,9) ..(8.5,8)
                    --(9.5,8)
                    .. controls (10.5,8)..(10.5,7);
                    \draw[shorten >=.2pt, shorten <=.2pt, line width=.8pt, line cap=round] (10.5,7).. controls (11.5,7)..(11.5,6);
                \end{tikzpicture}
            \]
            \caption{Correspondence of ray-shooting and the permutation. A certain corresponding segment is highlighted.}
            \label{fig:ray-shooting}
        \end{figure}
        
        The cardinality of $\mathcal{CP}(n)$ are listed in \cref{table:CP(n)}.
        \newpage
        
        \begin{table}[htbp!]
            \centering
            \caption{Number of elements in $\mathcal{CP}(n)$, classified by the number $k$ of flat edges in associated Motzkin paths}
            \label{table:CP(n)}
            \begin{tabular}{r|l|ccccccccccccccc}
            $n$\textbackslash $k$&total & 1&2&3&4&5&6&7&8&9&10&11&12&13&14\\
            1& 1  & 1\\
            2& 1  &  &1\\
            3& 2  & 1& &1\\
            4& 3  &  &2& &1\\
            5& 6  & 1& &4& &1\\
            6&11  &  &4& &6& &1\\
            7&22  & 1&  &11&&9&&1\\
            8&45  &  &8&&24&&12&&1\\
            9&92  & 1&&27&&47&&16&&1\\
           10&198 & &16&&80&&81&&20&&1\\
           11&418 & 1&&65&&192&&134&&25&&1\\
           12&922 & &32&&242&&412&&205&&30&&1\\
           13&2006& 1&&159&&691&&812&&306&&36&&1\\
           14&4490& &64&&694&&1773&&1482&&434&&42&&1
            \end{tabular}
        \end{table}

\backmatter


\end{document}